\documentclass[oneside, english]{article}
\usepackage[T1]{fontenc}
\usepackage[latin1]{inputenc}
\usepackage{amssymb}
\usepackage{amsmath}
\usepackage[all]{xy}
\usepackage{amsthm}
\usepackage{setspace}

\makeatletter

  \theoremstyle{definition}
  \newtheorem{defn}{Definition}[section]

  \theoremstyle{remark}
    
   \newtheorem{remark}[defn]{Remark}

   \theoremstyle{plain}
  \newtheorem{theorem}[defn]{Theorem}
  \newtheorem{lemma}[defn]{Lemma}

  \newtheorem{corollary}[defn]{Corollary}

    \newtheorem{example}[defn]{Example}
    \newcommand{\iso}{\overset{\sim}{\rightarrow}}

\usepackage{geometry}

\usepackage{babel}

\usepackage{amsmath,amssymb,amsfonts,latexsym}
\usepackage{mathrsfs}
\usepackage{epsf}
\usepackage{epsfig}
\usepackage[all]{xy}
\usepackage{enumerate}
\makeatother

\begin{document}

\def\AW[#1]^#2_#3{\ar@{-^>}@<.5ex>[#1]^{#2} \ar@{_<-}@<-.5ex>[#1]_{#3}}
 \def\NAW[#1]{\ar@{-^>}@<.5ex>[#1] \ar@{_<-}@<-.5ex>[#1]}

 \renewcommand{\arraystretch}{1}
 \renewcommand{\O}{\bigcirc}
 \newcommand{\OX}{\bigotimes}
 \newcommand{\OD}{\bigodot}
 \newcommand{\OV}{\O\llap{v\hspace{.6ex}}}
 \newcommand{\B}{\mbox{\Huge $\bullet$}}
  \newcommand{\D}{$\diamond$}

\title{Good gradings of basic Lie superalgebras}
\author{Crystal Hoyt\footnote{Department of Mathematics, Bar-Ilan University, Ramat Gan, 52900, Israel;
 hoyt@math.biu.ac.il.} \footnote{Supported by JSPS at Nara Women's University, Japan, and ISF at the Weizmann Institute of Science, Israel. Supported by the Minerva foundation with funding from the Federal German Ministry for Education and Research. }}

\date{}
\maketitle

\begin{abstract}
We classify good $\mathbb{Z}$-gradings of basic Lie superalgebras over an algebraically closed field $\mathbb{F}$ of characteristic zero. Good $\mathbb{Z}$-gradings are used in quantum Hamiltonian reduction for affine Lie superalgebras, where they play a role in the construction of super $W$-algebras. We also describe the centralizer of a nilpotent even element and of an $\mathfrak{sl}_{2}$-triple in $\mathfrak{gl}(m|n)$ and $\mathfrak{osp}(m|2n)$.
\end{abstract}

\setcounter{section}{-1}
\section{Introduction}\label{section0}

Good $\mathbb{Z}$-gradings of basic Lie superalgebras are used in the construction of super $W$-algebras, both finite and affine \cite{DK}.  In this paper, we classify good $\mathbb{Z}$-gradings of basic Lie superalgebras.  A finite-dimensional simple Lie superalgebra $\mathfrak{g}=\mathfrak{g}_{\bar{0}} \oplus \mathfrak{g}_{\bar{1}}$ is called basic if $\mathfrak{g}_{\bar{0}}$ is a reductive Lie algebra and there exists an even nondegenerate invariant bilinear form on $\mathfrak{g}$.   A $\mathbb{Z}$-grading $\mathfrak{g}=\oplus_{j\in\mathbb{Z}}\mathfrak{g}(j)$ is called good if there exists $e\in\mathfrak{g}_{\bar{0}}(2)$ such that the map $\mbox{ad }e:\mathfrak{g}(j)\rightarrow\mathfrak{g}(j+2)$ is injective for $j\leq -1$ and surjective for $j\geq -1$.
If a $\mathbb{Z}$-grading of $\mathfrak{g}$ is defined by a semisimple element $h\in\mathfrak{g}_{\bar{0}}$, then this condition is equivalent to all of the eigenvalues of $\mbox{ad}(h)$ on the centralizer $\mathfrak{g}^{e}$ of $e$ in $\mathfrak{g}$ being non-negative.

An example of a good $\mathbb{Z}$-grading for a nilpotent element $e\in\mathfrak{g}_{\bar{0}}$ is the Dynkin grading. By the Jacobson-Morosov Theorem, $e$ belongs to an $\mathfrak{sl}_2$-triple $\mathfrak{s}=\{e,f,h\}\subset\mathfrak{g}_{\bar{0}}$, where $[e,f]=h$, $[h,e]=2e$ and $[h,f]=-2f$. By $\mathfrak{sl}_2$ theory, the grading of $\mathfrak{g}$ defined by $\mbox{ad }h$ is a good $\mathbb{Z}$-grading for $e$.

Affine $W$-algebras are vertex algebras which can be realized using the homology of a BRST complex of a simple finite-dimensional Lie superalgebra $\mathfrak{g}$ with a non-degenerate even supersymmetric
invariant bilinear form. If $x$ is an ad-diagonalizable element of $\mathfrak{g}$ with half integer eigenvalues and if $f$ is an even nilpotent element of $\mathfrak{g}$ such that $[x,f]=-f$ and the eigenvalues of $\mbox{ad}(x)$ on the centralizer $\mathfrak{g}^{f}$ of $f$ in $\mathfrak{g}$ are all non-positive, then for each complex number $k$, one can define a vertex algebra $W^{k}(\mathfrak{g},x,f)$, as was shown by Kac, Roan and Wakimoto in 2003 \cite{KRW}.

The minimal W-algebras $W^{k}(\mathfrak{g},x,f_{\theta})$, where $f_{\theta}$ is a root vector of the lowest root $\theta$ (which is assumed to be even), have been studied more extensively \cite{KW04, KRW}.  This class of $W$-algebras contains the well known superconformal algebras.
Let $\widehat{\mathfrak{g}}$ be the (non-twisted) affinization of $\mathfrak{g}$ and let $O_k$ be the BGG-category of $\widehat{\mathfrak{g}}$ at level $k$.  A functor $H$ from the category $O_k$ to the category of integer graded modules of $W^{k}(\mathfrak{g},x,f_{\theta})$ was given by Kac, Roan and Wakimoto in \cite{KRW}. The quantum reduction functor has many nice properties, allowing one to transfer information between the two categories of modules. In particular, in 2005, Arakawa proved that this functor is exact and that the image of a simple highest weight module is either zero or irreducible \cite{A}.

A finite $W$-algebra is defined as follows \cite{Kos2, Prem, GG}.  Given a good $\mathbb{Z}$-grading $\mathfrak{g}=\oplus_{j\in\mathbb{Z}}\mathfrak{g}(j)$ for a nilpotent element $e\in\mathfrak{g}(2)$, choose an isotropic subspace $\mathfrak{l}$ of $\mathfrak{g}(-1)$ with respect to the skew-supersymmetric bilinear form defined by $\omega(x,y)= (e,[x,y])$. Let $\mathfrak{m}=\mathfrak{l}\oplus\bigoplus_{j\leq -2}\mathfrak{g}(j)$ and $\mathfrak{n}=\mathfrak{l}^{\bot}\oplus\bigoplus_{j\leq -2}\mathfrak{g}(j)$, where $\mathfrak{l}^{\bot}$ is the complement of $\mathfrak{l}$ with respect to $\omega$.  Define $\chi :\mathfrak{m}\rightarrow\mathbb{C}$ by $\chi(x)=(x,e)$, and let $\mathcal{Q}=U(\mathfrak{g})\otimes_{U(\mathfrak{m})}\mathbb{C}_{\chi}$.
The (super) {\em finite $W$-algebra} associated to $e$ for this choice of grading and isotropic subspace $\mathfrak{l}$ is defined to be $W^{fin}(\mathfrak{g},e)=\mathcal{Q}^{\mathrm{ad}\ n}$.  The algebra structure of $W^{fin}(\mathfrak{g},e)$ is induced from that of $U(\mathfrak{g})$.

Good $\mathbb{Z}$-gradings of simple finite-dimensional Lie algebras were classified by A.G. Elashvili and V.G. Kac in \cite{EK}.  K. Baur and N. Wallach classified nice parabolic subalgebras of reductive Lie algebras in \cite{BW}, which correspond to good even $\mathbb{Z}$-gradings by \cite[Theorem 2.1]{EK}.  J. Brundan and S. Goodwin classified good $\mathbb{R}$-gradings of semisimple Lie algebras in \cite{BG}, and proved that the isomorphism type of a (non-super) finite $W$-algebra does not depend on the choice of good grading.  W.L. Gan and V. Ginzburg proved that a (non-super) finite $W$-algebra does not depend on the choice of the isotropic subspace $\mathfrak{l}$ \cite{GG}.

The paper is organized as follows. In Section~\ref{section3}, we study $\mathbb{Z}$-gradings of basic Lie superalgebras. We obtain a criterion for when two diagram characteristics determine the same $\mathbb{Z}$-grading by using the action of the Weyl groupoid. In Section~\ref{section2}, we describe explicitly the centralizers of nilpotent even elements and of $\mathfrak{sl}_{2}$-triples in $\mathfrak{gl}(m|n)$ and $\mathfrak{osp}(m|2n)$. In Section~\ref{section4}, we establish some general results for good $\mathbb{Z}$-gradings of basic Lie superalgebras. We also examine the question of extending good $\mathbb{Z}$-gradings from $\mathfrak{g}_{\bar{0}}$ to $\mathfrak{g}$.
In Section~\ref{section5}, we prove that all good $\mathbb{Z}$-gradings of the exceptional Lie superalgebras $F(4)$, $G(3)$, and $D(2,1,\alpha)$ are Dynkin gradings. In Sections \ref{section6a}, \ref{section6} and \ref{section7}, we classify the good $\mathbb{Z}$-gradings of $\mathfrak{psl}(2|2)$,  $\mathfrak{gl}(m|n)$, and $\mathfrak{osp}(m|2n)$, respectively.  In particular, for each nilpotent even element (up to conjugacy) we describe all $\mathbb{Z}$-gradings for which the element is good.

We classify the good $\mathbb{Z}$-gradings of $\mathfrak{gl}(m|n)$ and $\mathfrak{osp}(m|2n)$ using pyramids.
Pyramids were defined in \cite{BG, EK} to describe the good $\mathbb{Z}$-gradings of $\mathfrak{gl}(n)$, $\mathfrak{so}(m)$ and $\mathfrak{sp}(2n)$.  We generalize these definitions to the Lie superalgebras $\mathfrak{gl}(m|n)$ and $\mathfrak{osp}(m|2n)$.
For $\mathfrak{gl}(m|n)$, a symmetric pyramid is defined for each nilpotent even element $e$ essentially by taking an $\mathfrak{sl}_2$-triple $\mathfrak{s}$ containing $e$ and then looking at the $\mathfrak{sl}_2$-strings in the standard representation of $\mathfrak{s}$.  One arranges rows of boxes in the upper half plane such that each row corresponds to an $\mathfrak{sl}_2$-string, the rows have non-increasing length in the positive $y$ direction, and the left coordinate of each box equals the weight of the vector to which it corresponds. Then by $\mathfrak{sl}_2$ theory this pyramid is symmetric about the $y$-axis. Any pyramid for  $\mathfrak{gl}(m|n)$ can be obtained from a symmetric pyramid by shifting the rows horizontally.  For $\mathfrak{osp}(m|2n)$, one adjusts the symmetric pyramid to contain a central symmetry about the origin.

\textbf{Acknowledgement.} I would like to thank Tomoyuki Arakawa for support and guidance; Vera Serganova for helpful conversations; Alexander Elashvili for suggesting good references; Anthony Joseph for thought-provoking discussions; and Maria Gorelik for reading drafts of the paper and suggesting improvements.

\section{Preliminaries}\label{section1}
We begin with recalling the definitions of  $\mathfrak{gl}(m|n)$ and  $\mathfrak{sl}(m|n)$.
Let $M_{r,s}$ denote the set of $r\times s$ matrices.  As a vector space $\mathfrak{gl}(m|n)$ is $M_{m+n,m+n}$, where
\begin{align*}
\mathfrak{g}_{\overline{0}}&=\left\{\left(\begin{array}{cc} A & 0 \\ 0 & B\end{array}\right)\mid A\in M_{m,m},\ B\in M_{n,n}\right\}\\
\mathfrak{g}_{\overline{1}}&=\left\{\left(\begin{array}{cc} 0 & C \\ D & 0\end{array}\right)\mid C\in M_{m,n},\ D\in M_{n,m}\right\}.
\end{align*}

The bracket operation is defined on homogeneous elements as follows:  if $X\in\mathfrak{g}_{i}$, $Y\in\mathfrak{g}_{j}$, then $[X,Y]:=XY-(-1)^{ij}YX,$ and it is extended linearly to the superalgebra.
The Lie superalgebra $\mathfrak{sl}(m|n)$ is defined to be $$\mathfrak{sl}(m|n)=\left\{X=\left(\begin{array}{cc} A & C \\ D & B\end{array}\right)\in\mathfrak{gl}(m|n)\mid \mbox{supertr}(X):=\mbox{tr}(A)-\mbox{tr}(B)=0 \right\}.$$

\subsection{Basic Lie superalgebras}

Finite-dimensional simple Lie superalgebras were classified by V.G. Kac in \cite{K77}.  These can be separated into three types: basic, strange and Cartan.  A finite-dimensional simple Lie superalgebra  $\mathfrak{g} = \mathfrak{g}_{\bar{0}} \oplus\mathfrak{g}_{\bar{1}}$ is called {\em basic} if $\mathfrak{g}_{\bar{0}}$ is a reductive Lie algebra and $\mathfrak{g}$ has an even nondegenerate invariant bilinear form $(\cdot,\cdot)$.  This form is necessarily supersymmetric.  These are: $\mathfrak{sl}(m|n) : m\neq n$, $\mathfrak{psl}(n|n):=\mathfrak{sl}(n|n)/ \langle I_{2n} \rangle$,  $\mathfrak{osp}(m|n)$, $D(2,1,\alpha)$, $F(4)$, $G(3)$, and finite dimensional simple Lie algebras.

\begin{equation*}\label{table1}\doublespacing
\begin{tabular}{ccccc}
& Table 1\\
\hline
$\mathfrak{g}$ &   & $\mathfrak{g}_{\bar{0}}$ & $\mathfrak{Z}(\mathfrak{g}_{\bar{0}})$ & $\kappa$  \\
  \hline
   $\mathfrak{sl}(m|n)$ &  $m,n\geq 1$, $m\neq n$ &  $\mathfrak{sl}(m) \times \mathfrak{sl}(n) \times \mathbb{C}$ & $\mathbb{C}$ &  \\
  $\mathfrak{psl}(n|n)$  & $n\geq 1$ &   $\mathfrak{sl}(n) \times \mathfrak{sl}(n)$ & $\{0\}$  & 0\\
  $\mathfrak{osp}(2|2n)$ & $n\geq 1$  & $\mathbb{C}\times \mathfrak{sp}(2n)$ & $\mathbb{C}$ & \\
   $\mathfrak{osp}(2n+2|2n)$  & $n\geq 1$  & $\mathfrak{so}(2n+2) \times \mathfrak{sp}(2n)$ & $\{0\}$ & 0 \\
   $\mathfrak{osp}(m|2n)$  & $m,n\geq 1$,  $m\neq 2, 2n+2$  & $\mathfrak{so}(m) \times \mathfrak{sp}(2n)$ & $\{0\}$ & \\
  $D(2,1,\alpha)$ & $\alpha\neq 0,-1$  & $\mathfrak{sl}(2) \times \mathfrak{sl}(2) \times \mathfrak{sl}(2)$  & $\{0\}$ & 0\\
  $F(4)$ &   & $\mathfrak{so}(7) \times \mathfrak{sl}(2)$ & $\{0\}$ & \\
  $G(3)$ &   & $G_2 \times \mathfrak{sl}(2)$ & $\{0\}$ & \\
  \hline
\end{tabular}
\end{equation*}

Note that if $\mathfrak{g}$ is a finite-dimensional simple Lie superalgebra, then $\mathfrak{g}_{\bar{0}}$ is a reductive if and only if the representation of $\mathfrak{g}_{\bar{0}}$ on $\mathfrak{g}_{\bar{1}}$ is completely reducible \cite{S}.  The Lie superalgebras $\mathfrak{sl}(m|n)$, $\mathfrak{osp}(m|2n)$, $D(2,1,\alpha)$, $F(4)$ and $G(3)$ are Kac-Moody superalgebras, i.e. they are defined by their Cartan matrix \cite{K77}.

Let $\mathfrak{g}$ be a basic Lie superalgebra. Elements of $\mathfrak{g}_{\bar{0}}$ are called {\em even}, while elements of $\mathfrak{g}_{\bar{1}}$ are called {\em odd}.
We can write $\mathfrak{g}_{\bar{0}} = \mathfrak{g}'_{\bar{0}} \times \mathfrak{Z}(\mathfrak{g}_{\bar{0}})$, where $\mathfrak{Z}(\mathfrak{g}_{\bar{0}})$ is the center of $\mathfrak{g}_{\bar{0}}$ and  $\mathfrak{g}'_{\bar{0}}:=[ \mathfrak{g}_{\bar{0}} , \mathfrak{g}_{\bar{0}} ]$ is semisimple.
If $\mathfrak{g} \neq \mathfrak{psl}(n|n), \mathfrak{osp}(2n+2|2n), D(2,1,\alpha)$ then
the Killing form $\kappa(x,y):=\mbox{supertr}((\mbox{ad }x)(\mbox{ad }y))$ is nondegenerate, and otherwise it is identically zero \cite{K77}.

For each $x\in\mathfrak{g}_{\bar{0}}$ the map $\mbox{exp}(\mbox{ad }x)$  is an automorphism of $\mathfrak{g}$. The group $G$ generated by these automorphisms is called the {\em group of inner automorphisms} of $\mathfrak{g}$.
Every inner automorphism of $\mathfrak{g}_{\bar{0}}$ extends to an inner automorphism of $\mathfrak{g}$ \cite{K77}.

\subsection{Decompositions of $\mathfrak{g}$}

A subalgebra $\mathfrak{h}\subset\mathfrak{g}$ is called a {\em Cartan subalgebra} if $\mathfrak{h}$ is nilpotent and $\mathfrak{h}$ equals its normalizer in $\mathfrak{g}$.  If $\mathfrak{g}$ is a basic Lie superalgebra, then $\mathfrak{h}$ is a Cartan subalgebra for $\mathfrak{g}$ if and only if it is a Cartan subalgebra for $\mathfrak{g}_{\bar{0}}$.
All Cartan subalgebras of $\mathfrak{g}$ are conjugate, because they are conjugate in the reductive Lie algebra $\mathfrak{g}_{\bar{0}}$.

Fix a Cartan subalgebra $\mathfrak{h}$. For $\alpha\in\mathfrak{h}^{*}$, let $\mathfrak{g}_{\alpha}:=\{x\in\mathfrak{g} \mid [h,x]=\alpha(h)x\ \text{ for all } h\in\mathfrak{h}\}$ and let $\Delta=\{\alpha\in\mathfrak{h}^{*}\setminus \{0\} \mid \mathfrak{g}_{\alpha}\neq 0\}$.  Then $\mathfrak{g}_{0}=\mathfrak{h}$ and $\mathfrak{g}$ has a root space decomposition $\mathfrak{g}=\mathfrak{h}\oplus\bigoplus_{\alpha\in\Delta} \mathfrak{g}_{\alpha}$.  The $\mathbb{Z}/2\mathbb{Z}$-grading of $\mathfrak{g}$ determines a decomposition of  $\Delta$ into the disjoint union of the even roots $\Delta_{\bar{0}}$ and the odd roots $\Delta_{\bar{1}}$.  Let $W$ denote the Weyl group of $\Delta_{\bar{0}}$.  Then $\Delta_{\bar{0}}$ and $\Delta_{\bar{1}}$ are invariant under $W$.

An element $h\in\mathfrak{h}$ is called {\em regular} if $\mbox{Re }\alpha(h)\neq 0$ for all $\alpha\in\Delta$.  A regular element $h\in\mathfrak{h}$ determines a decomposition of the roots $\Delta=\Delta^{+}\sqcup\Delta^{-}$ where
$\Delta^{+}:=\{\alpha\in\Delta \mid \mbox{Re } \alpha(h)>0\}$ and $ \Delta^{-}:=\{\alpha\in\Delta \mid \mbox{Re } \alpha(h)<0\}$.  This then determines a decomposition $\mathfrak{g}=\mathfrak{n}^{-}\oplus\mathfrak{h}\oplus\mathfrak{n}^{+}$ where $\mathfrak{n}^{+}:=\oplus_{\alpha\in\Delta^{+}}\mathfrak{g}_{\alpha}$ and
$\mathfrak{n}^{-}:=\oplus_{\alpha\in\Delta^{-}}\mathfrak{g}_{\alpha}$.
Such a decomposition is called a {\em triangular decomposition} \cite{PS94}.  We have an induced triangular decomposition $\mathfrak{g}_{\bar{0}}=\mathfrak{n}_{\bar{0}}^{-}\oplus\mathfrak{h}\oplus\mathfrak{n}_{\bar{0}}^{+}$
given by $\Delta_{\bar{0}}=\Delta^{+}_{\bar{0}}\sqcup\Delta^{-}_{\bar{0}}$. Corresponding to a decomposition $\Delta=\Delta^{+}\sqcup\Delta^{-}$, a {\em base} is a set of simple roots $\Pi\subset\Delta^{+}$ (resp. $\Pi_{\bar{0}}\subset\Delta_{\bar{0}}^{+}$) for $\mathfrak{g}$ (resp. $\mathfrak{g}_{\bar{0}})$.

A subalgebra $\mathfrak{b}\subset\mathfrak{g}$ is called a Borel subalgebra if $\mathfrak{b}=\mathfrak{h}\oplus\mathfrak{n}^{+}$ for some triangular decomposition $\mathfrak{g}=\mathfrak{n}^{-}\oplus\mathfrak{h}\oplus\mathfrak{n}^{+}$.
Since $\mathfrak{g}_{\bar{0}}$ is a reductive Lie algebra, the group of inner automorphisms of $\mathfrak{g}_{\bar{0}}$ acts transitively on the set of Borel subalgebras for $\mathfrak{g}_{\bar{0}}$.
Since every inner automorphism of $\mathfrak{g}_{\bar{0}}$ extends to an inner automorphism of $\mathfrak{g}$, the Borel subalgebras of $\mathfrak{g}_{\bar{0}}$ are conjugate in $\mathfrak{g}$.

\subsection{The bilinear form}\label{theform}
Let $\mathfrak{g}$ be a basic Lie superalgebra, and let $(\cdot,\cdot)$ be a nondegenerate invariant even supersymmetric bilinear form on $\mathfrak{g}$. Such a form is unique up to multiplication by a scalar \cite{K77}. There is an invariant even supersymmetric bilinear form on $\mathfrak{gl}(m|n)$, which when restricted to $\mathfrak{sl}(m|n)$ has kernel equal to the center of $\mathfrak{sl}(m|n)$ \cite{K77}. We will also denote this form by $(\cdot,\cdot)$.  Let $\mathfrak{h}\subset\mathfrak{g}$ be a Cartan subalgebra, and let $\Delta$ be the set of roots.

\begin{theorem}[V.G. Kac \cite{K77}]\label{classical}
If $\mathfrak{g}$ is a basic Lie superalgebra or if $\mathfrak{g}=\mathfrak{gl}(m|n)$ or $\mathfrak{sl}(n|n)$, then
\begin{enumerate}[(i)]
   \item $(\mathfrak{g}_{\alpha},\mathfrak{g}_{\beta})=0$ if $\alpha\neq -\beta$ for $\alpha,\beta\in\Delta\cup\{0\}$;
   \item  $(\cdot,\cdot)$ determines a nondegenerate pairing of $\mathfrak{g}_{\alpha}$ with $\mathfrak{g}_{-\alpha}$ for $\alpha\in\Delta$;
   \item  if $\mathfrak{g}\neq\mathfrak{gl}(m|n),\mathfrak{sl}(n|n)$, then the restriction of $(\cdot,\cdot)$ to $\mathfrak{h}$ is nondegenerate;
\item  if $\mathfrak{g}=\mathfrak{sl}(n|n)$, then the kernel of $(\cdot,\cdot)$ equals the center of $\mathfrak{g}$;
    \item  $[\mathfrak{g}_{\alpha},\mathfrak{g}_{\beta}]\neq 0$ if and only if $\alpha,\beta,\alpha + \beta\in\Delta\cup\{0\}$;
\item  if $\mathfrak{g}\neq\mathfrak{psl}(2|2)$, then $\mbox{dim }\mathfrak{g}_{\alpha}=1$ for all $\alpha\in\Delta$.
 \end{enumerate}
\end{theorem}

If we fix $(\cdot,\cdot)$, then we can use $(\cdot,\cdot)$ to identify $\mathfrak{h}$ with $\mathfrak{h}^{\star}$. Then $(\cdot,\cdot)$ is defined on $\Delta\subset\mathfrak{h}^{\star}$ through this identification.
A root $\alpha\in\Delta$ is called {\em isotropic} if $(\alpha,\alpha)=0$.  For a basic Lie superalgebra, a simple isotropic root is necessarily odd.

\subsection{Even and odd reflections}\label{reflections}
We recall the notion of odd reflections for basic Lie superalgebras \cite{LSS86}.

Two Borel subsuperalgebras $\mathfrak{b},\mathfrak{b}'\subset \mathfrak{g}$ are connected by an {\em odd reflection} along $\alpha_k$ if and only if $\alpha_k$ is a simple odd isotropic root of $\mathfrak{b}$ and
\begin{equation}\label{eqnodd} \Delta'^{+}=\left(\Delta^{+}\setminus\{\alpha_k\}\right)\cup\{-\alpha_k\}. \end{equation}
For the bases $\Pi\subset\Delta^{+}$ and $\Pi'\subset\Delta'^{+}$, we say that $\Pi'$ is obtained from $\Pi$ by an odd reflection with respect to $\alpha_k$.  This is defined explicitly on $\Pi$ by
\begin{equation}\label{oddrefl}
 r_{k}(\alpha_{i}):=  \left\{
  \begin{array}{ll}
    -\alpha_{k}, & \hbox{if $i=k$;} \\
    \alpha_{i}, & \hbox{if $a_{ik}=a_{ki}=0$, $i \neq k$;} \\
    \alpha_{i}+\alpha_{k}, & \hbox{if $a_{ik}\neq 0$ or $a_{ki} \neq 0$,  $i \neq k$;}
  \end{array}\hspace{1cm}\alpha_i\in\Pi.
\right.\end{equation}

If $\alpha_k\in\Pi$ is non-isotropic, then we define the (even) reflection $r_{k}$ at $\alpha_k$ by
\begin{equation}\label{evenrefl}
r_{k}(\alpha)=\beta - \frac{2(\beta,\alpha_k)}{(\alpha_k,\alpha_k)}\alpha_k\hspace{1cm}\beta\in\Delta.\end{equation}
If $\alpha_k\in\Pi$ is an even root, then $r_{k}$ also satisfies (\ref{eqnodd}).  However, if $\alpha_k\in\Pi$ is a non-isotropic odd root, then
\begin{equation}\label{eqneven}\Delta'^{+}=\left(\Delta^{+}\setminus\{\alpha_k,2\alpha_k\}\right)\cup\{-\alpha_k,-2\alpha_k\}.\end{equation}

\subsection{The Weyl groupoid}\label{groupoid}

The {\em Weyl groupoid} $\mathcal{W}$ for a basic Lie superalgebra $\mathfrak{g}$ is a groupoid which acts by even and odd reflections on the set of bases of $\mathfrak{g}$ \cite{S08}.  For each base $\Pi=\{\alpha_1,\ldots,\alpha_n\}$ and each simple root $\alpha_k\in\Pi$ the set $r_k(\Pi)\subset\Delta$ defined by $r_k(\Pi)=\{r_{k}(\alpha_{n}),\ldots,r_{k}(\alpha_{n})\}$ is a base for $\Delta$ \cite{S08}.  The Weyl groupoid acts transitively on the set of bases of a basic Lie superalgebra.  Indeed, if we have two different decompositions $\Delta=\Delta^{+}\sqcup\Delta^{-}$ and $\Delta=\Delta''^{+}\sqcup\Delta''^{-}$, then there is a simple root $\alpha_k\in\Delta^{+}\setminus\Delta''^{+}$.  Let $\Delta'^{+}$ be obtained from $\Delta^{+}$ by the simple reflection $r_k$.  Then by (\ref{eqnodd}) and (\ref{eqneven}), $|\Delta'^{+}\setminus\Delta''^{+}| < |\Delta^{+}\setminus\Delta''^{+}|$, so the claim follows by induction.  If $\Delta_{\bar{0}}^{+}=\Delta_{\bar{0}}''^{+}$, then $\Delta^{+}\setminus\Delta''^{+}$ consists entirely of odd roots. So any two Borel subalgebras $\mathfrak{b},\mathfrak{b}'\subset\mathfrak{g}$ satisfying $\mathfrak{b}_{\bar{0}}=\mathfrak{b}'_{\bar{0}}$ are connected by a chain of odd reflections.  In particular, one can use odd reflections to move between the different Dynkin diagrams of a basic Lie superalgebra.

The Weyl groupoid also acts on the set of all ``marked bases'' of $\mathfrak{g}$.  A {\em marked base} $(\Pi,D)$ is a base $\Pi=\{\alpha_1,\ldots,\alpha_n\}$ together with an assignment of integers $D=\{d_1,\ldots,d_n\}$. Given a marked base $(\Pi,D)$, we can extend $D$ linearly to a map $D:\Delta\cup\{0\}\rightarrow\mathbb{Z}$ by $D(\beta)=D(\sum_{i=1}^{n}k_i\alpha_i)=\sum_{i=1}^{n}k_i d_i$.  For each simple root $\alpha_k\in\Pi$, we can reflect at $\alpha_k$ to obtain a marked base $(r_k(\Pi),r_k(D))$, where $r_k(D)=\{D(r_k(\alpha_1)),\ldots,D(r_k(\alpha_n))\}$. For each $i=1,\ldots,n$, $D(r_k(\alpha_i))$ can be easily computed from the definition of $r_k$ using linearity.
If $D(\alpha_k)=0$, then $D(r_k(\alpha_i))=D(\alpha_i)$ for $i=1,\ldots,n$.
Similarly, $\mathcal{W}$ acts on the set of all ``marked diagrams'' of $\mathfrak{g}$.
A {\em marked diagram} is obtained by assigning an integer to each vertex of a Dynkin diagram of $\mathfrak{g}$.

\section{Properties of $\mathbb{Z}$-gradings}\label{section3}

A $\mathbb{Z}$-grading of a Lie superalgebra $\mathfrak{g}$ is a decomposition
into a direct sum of $\mathbb{Z}/2\mathbb{Z}$-graded subspaces $ \mathfrak{g}=\oplus_{j\in\mathbb{Z}} \mathfrak{g}(j)$ such that $[\mathfrak{g}(i),\mathfrak{g}(j)]\subset\mathfrak{g}(i+j)$.
Let $$\mathfrak{g}_{\geq}=\oplus_{j\geq 0} \mathfrak{g}(j),\ \ \mathfrak{g}_{\leq}=\oplus_{j\leq 0} \mathfrak{g}(j),\ \ \mathfrak{g}_{+}=\oplus_{j>0} \mathfrak{g}(j),\ \text{ and }\ \mathfrak{g}_{-}=\oplus_{j<0} \mathfrak{g}(j).$$

If $\mathfrak{a}$ is a subalgebra of a $\mathbb{Z}$-graded Lie superalgebra $\mathfrak{g}=\oplus_{j\in\mathbb{Z}} \mathfrak{g}(j)$, we define $\mathfrak{a}(k):=\mathfrak{a}\cap\mathfrak{g}(k)$.
Then $\oplus_{k\in\mathbb{Z}}\mathfrak{a}(k)$ is a subalgebra of $\mathfrak{a}$. We say that that $\mathfrak{a}$ is a {\em graded subalgebra of $\mathfrak{g}$} if $\mathfrak{a}=\oplus_{k\in\mathbb{Z}}\mathfrak{a}(k)$, and call this grading the {\em induced grading}.
Clearly, $\mathfrak{g}_{\bar{0}}$ is a graded subalgebra of $\mathfrak{g}$.
One can show that the derived subalgebra $\mathfrak{g}'$ and the center $\mathfrak{Z}(\mathfrak{g})$ are also graded subalgebras of $\mathfrak{g}$.  Moreover the {\em centralizer of $T$ in $\mathfrak{g}$}, defined by  $\mathfrak{g}^{T}:=\{x\in\mathfrak{g}\mid [x,t]=0\ \forall t\in T\}$, is a graded subalgebra of $\mathfrak{g}$ when $T$ is spanned by set of homogeneous elements.

\subsection{Cartan subalgebras and the root space decomposition}

\begin{lemma}\label{CSA}
Let $\mathfrak{g}$ be a basic Lie superalgebra, or let $\mathfrak{g}$ be $\mathfrak{gl}(m|n)$ or $\mathfrak{sl}(n|n)$.  Let  $\mathfrak{g}=\oplus_{j\in\mathbb{Z}} \mathfrak{g}(j)$ be a $\mathbb{Z}$-grading of $\mathfrak{g}$ satisfying $\mathfrak{Z}(\mathfrak{g}_{\bar{0}})\subset \mathfrak{g}_{\bar{0}}(0)$.  Then
\begin{enumerate}[(i)]
  \item $\mathfrak{g}'_{\bar{0}}(0)$ is a reductive Lie algebra and so is $\mathfrak{g}_{\bar{0}}(0)=\mathfrak{g}'_{\bar{0}}(0) \times \mathfrak{Z}(\mathfrak{g}_{\bar{0}})$;
  \item there exists a Cartan subalgebra $\mathfrak{h}$ for $\mathfrak{g}$ such that $\mathfrak{h}\subset\mathfrak{g}_{\bar{0}}(0)$.
\end{enumerate}
\end{lemma}
\begin{proof}
Now $\mathfrak{g}_{\bar{0}}=\mathfrak{g}'_{\bar{0}} \times \mathfrak{Z}(\mathfrak{g}_{\bar{0}})$ is reductive,
and $\mathfrak{g}_{\bar{0}}(0)=\mathfrak{g}'_{\bar{0}}(0) \times \mathfrak{Z}(\mathfrak{g}_{\bar{0}})$.
Since $\mathfrak{g}'_{\bar{0}}$ is a semisimple Lie algebra, there exists $H\in\mathfrak{g}'_{\bar{0}}(0)$ which defines the induced grading of  $\mathfrak{g}'_{\bar{0}}$.  Now $[H,\mathfrak{Z}(\mathfrak{g}_{\bar{0}})]=0$ and by assumption $\mathfrak{Z}(\mathfrak{g}_{\bar{0}})\subset\mathfrak{g}_{\bar{0}}(0)$, hence $H$ defines the induced grading of $\mathfrak{g}_{\bar{0}}$.

It is well known that if $\mathfrak{a}=\oplus_{j\in\mathbb{Z}} \mathfrak{a}(j)$ is a $\mathbb{Z}$-grading of a semisimple Lie algebra and $h\in\mathfrak{a}$ defines the grading, then $\mathfrak{a}(0)=\mathfrak{a}^{h}$ is a reductive Lie algebra. In particular, $\mathfrak{g}'_{\bar{0}}(0)=(\mathfrak{g}'_{\bar{0}})^{H}$ is a reductive Lie algebra.
Moreover, $\mathfrak{g}_{\bar{0}}(0)=\mathfrak{g}'_{\bar{0}}(0)\times\mathfrak{Z}(\mathfrak{g}_{\bar{0}})$ is a reductive Lie algebra.

Since $H$ is semisimple in $\mathfrak{g}_{\bar{0}}$ we can choose a Cartan subalgebra $\mathfrak{h}$ in $\mathfrak{g}_{\bar{0}}(0)$ which contains $H$.  Then $\mathfrak{h}$ is a Cartan subalgebra for $\mathfrak{g}_{\bar{0}}$, and so by \cite{K77}  $\mathfrak{h}$ is a Cartan subalgebra for $\mathfrak{g}$.
\end{proof}

\begin{remark}
If $\mathfrak{g}$ is a basic Lie superalgebra with a $\mathbb{Z}$-grading $\mathfrak{g}=\oplus_{j\in\mathbb{Z}} \mathfrak{g}(j)$, then $\mathfrak{Z}(\mathfrak{g}_{\bar{0}})\subset \mathfrak{g}(0)$.
Indeed, elements of  $\mathfrak{Z}(\mathfrak{g}_{\bar{0}})$ are ad-semisimple on $\mathfrak{g}$ since the action of $\mathfrak{g}_{\bar{0}}$ on $\mathfrak{g}_{\bar{1}}$ is completely reducible, whereas elements of $\mathfrak{g}(j)$ for $j\neq 0$ are ad-nilpotent on $\mathfrak{g}$. The claim follows since $\mathfrak{Z}(\mathfrak{g})=0$.
\end{remark}

\begin{lemma}\label{rootdecomposition}\label{degree}
Let $\mathfrak{g}$ be a basic Lie superalgebra, $\mathfrak{g}\neq\mathfrak{psl}(2|2)$, or let $\mathfrak{g}$ be $\mathfrak{gl}(m|n)$ or $\mathfrak{sl}(n|n)$.  Fix a $\mathbb{Z}$-grading  $\mathfrak{g}=\oplus_{j\in\mathbb{Z}} \mathfrak{g}(j)$ satisfying $\mathfrak{Z}(\mathfrak{g}_{\bar{0}})\subset \mathfrak{g}_{\bar{0}}(0)$.
Fix a Cartan subalgebra $\mathfrak{h}\subset\mathfrak{g}_{\bar{0}}(0)$.  Let $\mathfrak{g}=\mathfrak{h}\oplus\bigoplus_{\alpha\in\Delta}\mathfrak{g}_{\alpha}$ be the corresponding root space decomposition of $\mathfrak{g}$.

\begin{enumerate}[(i)]
\item
 For each $\alpha\in\Delta$, there exist $k\in\mathbb{Z}$ such that $\mathfrak{g}_{\alpha}\subset\mathfrak{g}(k)$. Thus \begin{equation*}\label{deco} \mathfrak{g}(0)=\mathfrak{h}\oplus\bigoplus_{\alpha\in\Delta_{0}} \mathfrak{g}_{\alpha}\ \text{ and }\ \mathfrak{g}(j)=\bigoplus_{\alpha\in\Delta_{j}} \mathfrak{g}_{\alpha}\ \text{ for each }\ j\in\mathbb{Z},\ j\neq 0, \end{equation*}
where $\Delta_{j}=\{\alpha\in\Delta \mid \mathfrak{g}_{\alpha}\subset\mathfrak{g}_j\}.$

\item
Define the degree map $\mbox{Deg}:\Delta\cup\{0\}\rightarrow\mathbb{Z}$ by $\mbox{Deg}(\alpha)= k$ if $\alpha\in\Delta_k$ and $\mbox{Deg}(0)=0$.  Then
$\mbox{Deg}$ is a linear map, $\mbox{Deg}(-\alpha)=-\mbox{Deg}(\alpha)$ for all $\alpha\in\Delta$, and $\Delta_{-j} =-\Delta_{j}$ for all $j\in\mathbb{Z}$.
\end{enumerate}
\end{lemma}

\begin{proof}
Fix $\alpha\in\Delta$ and let $x\in\mathfrak{g}_{\alpha}$, $x\neq 0$.  Write $x=\sum_{j\in\mathbb{Z}} x_j$ where $x_j \in\mathfrak{g}(j)$. Then for each $h\in\mathfrak{h}$  we have that
$$ \sum_{j\in\mathbb{Z}} [h,x_j]= [h,x] = \alpha (h) x =  \sum_{j\in\mathbb{Z}} \alpha (h) x_j$$
Since $\mathfrak{h}$ preserves each graded component, this implies $[h,x_j]=\alpha (h) x_j$ for each $h\in\mathfrak{h}$ and $j\in\mathbb{Z}$. Thus $x_j\in\mathfrak{g}_{\alpha}$ for each $j\in\mathbb{Z}$.
For $\mathfrak{g}\neq \mathfrak{psl}(2|2)$,  $\mbox{dim }\mathfrak{g}_{\alpha}=1$ implies that $x=x_k\in\mathfrak{g}(k)$ for some $k\in\mathbb{Z}$.  Hence, $\mathfrak{g}_{\alpha}=\mathbb{C}x\subset\mathfrak{g}(k)$ for some $k\in\mathbb{Z}$.
Now  $[\mathfrak{g}_{\alpha},\mathfrak{g}_{\beta}]\subset\mathfrak{g}_{\alpha+\beta}$,  and by Theorem~\ref{classical}, $[\mathfrak{g}_{\alpha},\mathfrak{g}_{\beta}]\neq 0$ when  $\alpha,\beta,\alpha+\beta\in\Delta\cup\{0\}$. Thus,  $\mbox{Deg}(\alpha)+\mbox{Deg}(\beta)=\mbox{Deg}(\alpha +\beta)$ for  $\alpha,\beta,\alpha+\beta\in\Delta\cup\{0\}$.

\end{proof}

\subsection{Inner derivations}

Let $\mathfrak{g}=\oplus_{j\in\mathbb{Z}}\mathfrak{g}(j)$ be a $\mathbb{Z}$-grading of a Lie superalgebra $\mathfrak{g}$. The linear map $\phi:\mathfrak{g} \rightarrow \mathfrak{g}$ defined by $\sum_{j\in\mathbb{Z}} x_j \mapsto\sum_{j\in\mathbb{Z}} jx_j$, with $x_j\in \mathfrak{g}(j)$, is a derivation.
If $\mathfrak{g}$ is a semisimple Lie algebra or a basic Lie superalgebra, $\mathfrak{g}\neq \mathfrak{psl}(n|n)$, then all derivations of $\mathfrak{g}$ are inner \cite{K77, S}. So there exists $H\in \mathfrak{g}$ that defines the grading, that is, $ [H,x_j]=jx_j$ for all $ x_j\in\mathfrak{g}(j)$.  Since $\phi$ preserves parity, $H\in \mathfrak{g}_{\bar{0}}$.

\begin{lemma}
A $\mathbb{Z}$-grading of $\mathfrak{g}=\mathfrak{sl}(n|n)$ or $\mathfrak{gl}(n|n)$ which satisfies $\mathfrak{Z}(\mathfrak{g}_{\bar{0}})\subset\mathfrak{g}(0)$ is defined by an inner derivation of $\mathfrak{gl}(n|n)$.
\end{lemma}

\begin{proof}
We can extend a $\mathbb{Z}$-grading of $\mathfrak{sl}(n|n)$ to a $\mathbb{Z}$-grading of $\mathfrak{g}=\mathfrak{gl}(n|n)$  such that $\mathfrak{Z}(\mathfrak{g}_{\bar{0}})\subset\mathfrak{g}(0)$. Fix a Cartan subalgebra $\mathfrak{h}\subset\mathfrak{g}(0)$.  By Lemma~\ref{rootdecomposition}, the $\mathbb{Z}$-grading is compatible with the root space decomposition.  A $\mathbb{Z}$-grading is determined the value of the degree map on a set of simple roots $\Pi=\{\alpha_1,\ldots,\alpha_{2n-1}\}\subset\Delta$. Since $\Pi$ is a linearly independent set in $\mathfrak{h}^{*}$, there exists $H\in\mathfrak{h}$ such that $\alpha_i(H)=\mbox{Deg}(\alpha_i)$ for $i=1,\ldots,2n-1$.  Clearly, $\mbox{ad }H$ defines the $\mathbb{Z}$-grading of $\mathfrak{sl}(n|n)\subset\mathfrak{gl}(n|n)$.
\end{proof}

\subsection{$\mathbb{Z}$-gradings of $\mathfrak{psl}(n|n)$}

\begin{lemma}\label{indu}
For $n\neq 2$, a $\mathbb{Z}$-grading of $\mathfrak{psl}(n|n)$ is induced from a $\mathbb{Z}$-grading of $\mathfrak{sl}(n|n)$, which satisfies $\mathfrak{Z}(\mathfrak{sl}(n|n)_{\bar{0}})\subset\mathfrak{sl}(n|n)(0)$.
\end{lemma}

\begin{proof}
Fix a $\mathbb{Z}$-grading of $\mathfrak{psl}(n|n)$,
and a Cartan subalgebra $\mathfrak{h}\subset\mathfrak{g}(0)$.
Let $\Pi=\{\alpha_1,\ldots,\alpha_n\}$ be a set of simple roots.
By Lemma~\ref{rootdecomposition}, $e_i\in\mathfrak{g}_{\alpha_i}$ and $f_i\in\mathfrak{g}_{-\alpha_i}$ are homogeneous for $i=1,\ldots,n$, and $\mbox{deg}(-\alpha_i)=-\mbox{deg}(\alpha_i)$.
Let $d_i=\mbox{Deg}(\alpha_i)$ for $i=1,\dots,n$.  This determines uniquely a $\mathbb{Z}$-grading of  $\mathfrak{sl}(n|n)$ with $\mathfrak{Z}(\mathfrak{sl}(n|n))\subset\mathfrak{sl}(n|n)(0)$.  Since $\mathfrak{psl}(n|n)$ is generated by $e_1,\dots,e_n,f_1,\ldots,f_n$, the quotient of this $\mathbb{Z}$-grading of $\mathfrak{sl}(n|n)$ by the center must coincide with the original $\mathbb{Z}$-grading of $\mathfrak{psl}(n|n)$.  In particular, $\mathfrak{psl}(n|n)(j)=\mathfrak{sl}(n|n)(j)$ for all $j\in\mathbb{Z}\setminus\{0\}$ and $\mathfrak{psl}(n|n)(0)=\mathfrak{sl}(n|n)(0)/\mathfrak{Z}(\mathfrak{sl}(n|n))$.
\end{proof}

The following lemma can be proved by explicit computation.

\begin{lemma}\label{rempsl}
The $\mathbb{Z}$-gradings of $\mathfrak{g}=\mathfrak{psl}(2|2)$ (up to conjugation by $\mathfrak{g}_{\bar{0}}$) are parameterized as follows.
For each $m\in\mathbb{Z}$, $p,q\in\{0,2\}$, and $(a:b), (c:d)\in\mathbb{P}^{2}$ satisfying $(a:b) \neq (c:d)$ we obtain a $\mathbb{Z}$-grading from the following assignment of degrees to the linear basis (of representatives modulo the center $\mathbb{C}(I_{4})$):
$$\begin{array}{lllll}
\mbox{Deg}(E_{12})=  p &  &\mbox{Deg}(aE_{14}+bE_{32})=     m  &  &\mbox{Deg}(bE_{31}-aE_{24})= m-p\\
\mbox{Deg}(E_{34})=  q  &  &\mbox{Deg}(dE_{41}+cE_{23})=    -m  & &\mbox{Deg}(cE_{13}-dE_{42})= p-m\\
\mbox{Deg}(E_{21})= -p &  &\mbox{Deg}(cE_{14}+dE_{32})=     p+q-m & &\mbox{Deg}(dE_{31}-cE_{24})= q-m\\
\mbox{Deg}(E_{43})= -q  &  &\mbox{Deg}(bE_{41}+aE_{23})=    m-p-q & &\mbox{Deg}(aE_{13}-bE_{42})= m-q\\
\mbox{Deg }(E_{11}+E_{33})=0 & & \mbox{Deg }(E_{22}+E_{44})=0 &&\\
\end{array}$$
\end{lemma}

\subsection{The bilinear form}

\begin{lemma}\label{pairing}  Let $\mathfrak{g}$ be a basic Lie superalgebra, or let $\mathfrak{g}$ be $\mathfrak{gl}(m|n)$ or $\mathfrak{sl}(n|n)$ with $(\cdot,\cdot)$ as defined in Section~\ref{theform}.
Fix a $\mathbb{Z}$-grading  $\mathfrak{g}=\oplus_{j\in\mathbb{Z}} \mathfrak{g}(j)$ satisfying $\mathfrak{Z}(\mathfrak{g})\subset\mathfrak{g}(0)$. Then
\begin{enumerate}[(i)]
  \item $(\mathfrak{g}(i),\mathfrak{g}(j))=0$ for $i\neq -j$;
  \item $(x,\mathfrak{g}(j))\neq 0$ for $x\in\mathfrak{g}_{-j}$, $x\not\in \mathfrak{Z}(\mathfrak{g})$.
\end{enumerate}
\end{lemma}
\begin{proof}
This follows from Theorem~\ref{classical}, Lemma~\ref{degree} and Lemma~\ref{rempsl}.
\end{proof}

\subsection{Characteristics and the action of the Weyl groupoid}
We can represent a $\mathbb{Z}$-grading by the values of the degree map on a set of simple roots. In this section, we examine the properties of the degree map with respect to different sets of simple roots.

Let $\mathfrak{g}$ be a basic Lie superalgebra, $\mathfrak{g}\neq\mathfrak{psl}(2|2)$, or let $\mathfrak{g}$ be $\mathfrak{gl}(m|n)$ or $\mathfrak{sl}(n|n)$.  Fix a $\mathbb{Z}$-grading  $\mathfrak{g}=\oplus_{j\in\mathbb{Z}} \mathfrak{g}(j)$ satisfying $\mathfrak{Z}(\mathfrak{g}_{\bar{0}})\subset \mathfrak{g}_{\bar{0}}(0)$.
Fix a Cartan subalgebra $\mathfrak{h}\subset\mathfrak{g}_{\bar{0}}(0)$, and let $\Delta$ be the set of roots.
Let  $\Delta_{\geq}=\{\alpha\in\Delta| \mbox{Deg}(\alpha)\geq 0\}$ and $\Delta_{<}=\{\alpha\in\Delta| \mbox{Deg}(\alpha)<0\}$.

\begin{lemma}\label{positive}
There exists a Borel subalgebra  $\mathfrak{b}\subset\mathfrak{g}_{\geq}$,  $\Delta^{+}\subset\Delta_{\geq}$. In particular, $\mathfrak{g}_{\geq}$ is a parabolic subalgebra with nilradical $\mathfrak{g}_{+}$ and Levi subalgebra $\mathfrak{g}(0)$.
\end{lemma}

\begin{proof} Fix a $\mathbb{Z}$-grading $\mathfrak{g}=\oplus_{j\in\mathbb{Z}} \mathfrak{g}(j)$
with $\mathfrak{Z}(\mathfrak{g}_{\bar{0}})\subset\mathfrak{g}(0)$.
Let $\Pi$ be a base of $\Delta$ and $\Delta=\Delta^{+}\sqcup\Delta^{-}$.  If $\Delta^{+}\not\subset\Delta_{\geq}$, then there is $\alpha_k\in\Pi\cap \Delta_{<}$. Let $r_k$ denote the reflection of $\Delta$ with respect to $\alpha_k$ (see Section~\ref{reflections}).  Then $\Pi':=\{r_k(\alpha_1),\ldots,r_k(\alpha_n)\}$ is a base for $\mathfrak{g}$ with decomposition $\Delta'=\Delta'^{+}\sqcup\Delta'^{-}$.   By (\ref{eqnodd}), (\ref{eqneven}) and Lemma~\ref{degree},
$|\Delta'^{+}\cap\Delta_{<}|=|\Delta^{+}\cap\Delta_{<}|-|\{\alpha_k,2\alpha_k\}| < |\Delta^{+}\cap\Delta_{<}|$. Since $\Delta$ is finite, the claim follows by induction.
\end{proof}

It may be possible to choose more than one Borel subalgebra in $\mathfrak{g}_{\geq}$. However, we have

\begin{lemma}\label{zero}
Suppose $\Delta^{+}_1,\Delta^{+}_2\subset\Delta_{\geq}$. Let $\gamma\in\Delta_1^{+}\setminus\Delta_2^{+}$. Then $\mbox{Deg}(\gamma)=-\mbox{Deg}(-\gamma)=0$.
\end{lemma}

\begin{proof}
By  Lemma~\ref{degree}, $-\gamma\in\Delta_2^{+}$.  So $\pm\gamma\in\Delta_{\geq}$ implying $\mbox{Deg}(\gamma)=-\mbox{Deg}(-\gamma)=0$.
\end{proof}

Now for each base $\Pi\subset\Delta$, the degree map of a $\mathbb{Z}$-grading is determined by its restriction to $\Pi$, that is, by $D:\Pi\rightarrow\mathbb{Z}$.  A reflection at a simple root of $\Pi$ yields a new map $D':\Pi'\rightarrow\mathbb{Z}$, where $\Pi'$ is the reflected base and $D'$ is defined on $\Pi'$ by linearity (see Section~\ref{groupoid}).  The maps $D:\Pi\rightarrow\mathbb{Z}$ and  $D':\Pi'\rightarrow\mathbb{Z}$ define the same grading on $\Delta$. Moreover, any map $D':\Pi'\rightarrow\mathbb{Z}$ obtained from $D:\Pi\rightarrow\mathbb{Z}$ by the action of the Weyl groupoid $\mathcal{W}$ defines the same grading on $\Delta$.

If $\Pi\subset\Delta_{+}$, then the induced map $\mbox{Deg}:\Pi\rightarrow\mathbb{N}$ is called the {\em characteristic} of the $\mathbb{Z}$-grading with respect to $\Pi$. It is natural to ask the following question:
when do two maps $D_1:\Pi_1\rightarrow\mathbb{N}$ and $D_2:\Pi_2\rightarrow\mathbb{N}$ define the same $\mathbb{Z}$-grading, i.e. when can they be extended to a linear map $\mbox{Deg}:\Delta\cup\{0\}\rightarrow\mathbb{Z}$?

\begin{theorem}
Let $\Pi_1=\{\alpha_1,\ldots,\alpha_n\},\ \Pi_2=\{\beta_1,\ldots,\beta_n\}$ be two different bases for $\Delta$.
If the maps $D_1:\Pi_1\rightarrow\mathbb{N}$ and $D_2:\Pi_2\rightarrow\mathbb{N}$ define the same grading, then there is a sequence of even and odd reflections $\mathcal{R}$ at simple roots of degree zero such that (after reordering) $\mathcal{R}(\alpha_i)=\beta_i$ and $D_1(\alpha_i)=D_2(\beta_i)$ for $i=1,\ldots,n$.
\end{theorem}

\begin{proof}
Suppose that $\mbox{Deg}:\Delta\cup\{0\}\rightarrow\mathbb{Z}$ is a linear map whose restriction to $\Pi_1$  is $D_1$ and to  $\Pi_2$ is $D_2$.
Let $\alpha_k\in\Pi_1$ such that $\alpha_k\not\in\Delta_2^{+}$, and let $r_k$ be the reflection at $\alpha_k$. By Lemma~\ref{zero}, $\mbox{Deg}(\alpha_k)=0$ which implies
$\mbox{Deg}(r_k(\alpha_i))=\mbox{Deg}(\alpha_i)$ for $i=1,\ldots ,n$. Let $\Pi':=\{r_k(\alpha_{1}),\ldots ,r_k(\alpha_{n})\}$ and let $\Delta'=\Delta'^{+}\sqcup\Delta'^{-}$ be the  corresponding decomposition. Then $|\Delta'^{+}\setminus\Delta_2^{+}| < |\Delta_1^{+}\setminus\Delta_2^{+}|$.
Since $\Delta$ is finite, it follows by
induction that there is a sequence of reflections $\mathcal{R}$ at simple roots of degree zero such that $R(\alpha_i)=\beta_i$ and $D_1(\alpha_i)=\mbox{Deg}(\alpha_i)=\mbox{Deg}(\beta_i)=D_2(\beta_i)$ for $i=1,\ldots,n$.
\end{proof}

Given a map $D:\Pi\rightarrow\mathbb{N}$, an even reflection at a simple root of positive degree yields a map $D':\Pi'\rightarrow\mathbb{Z}$ with $\mbox{Im }D'\not\subset\mathbb{N}$. Whereas, an even reflection at a simple root of degree zero does not change $D$ or the Dynkin diagram corresponding to $\Pi$.  Thus, a $\mathbb{Z}$-grading is determined up to equivalence by a Dynkin diagram $\Gamma$ and a labeling the vertices of $\Gamma$ by nonnegative integers.  Hence, for a simple Lie algebra $\mathfrak{g}$, there is a bijection between all $\mathbb{Z}$-gradings of $\mathfrak{g}$ up to conjugation and all characteristics of the Dynkin diagram \cite{EK}.  However, a Lie superalgebra usually has more than one Dynkin diagram.

Two Dynkin diagrams $\Gamma_1,\Gamma_2$ for a basic Lie superalgebra $\mathfrak{g}$ with degree maps $D_i:\Gamma_i\rightarrow\mathbb{N}$ define the same $\mathbb{Z}$-grading if and only if there is a sequence of odd reflections $\mathcal{R}$ at simple isotropic roots of degree zero such that $\mathcal{R}(\Gamma_1)=\Gamma_2$ and $D_1=D_2$ with the ordering of the vertices defined by $\mathcal{R}$.  This defines an equivalence relation on Dynkin diagrams with nonnegative integer labels.  We note that if a marked diagram has no isotropic roots of degree zero, then there is only one member in its equivalence class.

\section{Centralizers in basic Lie superalgebras}\label{section2}

\subsection{Nilpotent even elements}
Let $\mathfrak{g}$ be a basic Lie superalgebra, or let $\mathfrak{g}$ be $\mathfrak{gl}(m|n)$ or $\mathfrak{sl}(n|n)$.
In this section we discuss the orbits of nilpotent even elements  in $\mathfrak{g}$ under the action of the group of inner automorphisms $G$.  Recall that $G$ is the group of automorphisms of $\mathfrak{g}$ generated by $\mbox{exp}(\mbox{ad }x)$ for $x\in\mathfrak{g}_{\bar{0}}$.  For $x\in\mathfrak{g}$, let $Gx$ denote the {\em orbit} of $x$ in $\mathfrak{g}$ under the action of $G$.  An element $e\in\mathfrak{g}$ is called {\em nilpotent} if the action of $\mbox{ad }e$ on $\mathfrak{g}$ is nilpotent.

Let $e\in\mathfrak{g}$ be a nilpotent even element.  Then $e\in\mathfrak{g}_{\bar{0}}'$, since
elements of $\mathfrak{Z}(\mathfrak{g}_{\bar{0}})$ are semisimple in $\mathfrak{g}$ \cite{K77}. So $Ge\subset\mathfrak{g}_{\bar{0}}'$.  Let $G'$ be the group of automorphisms of $\mathfrak{g}_{\bar{0}}'$ generated by $\mbox{exp}(\mbox{ad }e)$  for $e\in\mathfrak{g}_{\bar{0}}'$. It follows that $Ge=G'e\subset\mathfrak{g}_{\bar{0}}'$.
Moreover, if $e\in\mathfrak{g}_{\bar{0}}'$ is ad-nilpotent on $\mathfrak{g}_{\bar{0}}'$, then $e$ is ad-nilpotent on $\mathfrak{g}$. This follows from the Jacobson-Morosov Theorem and $\mathfrak{sl}_2$ theory since $\mathfrak{g}$ is finite dimensional. Thus we are reduced to studying nilpotent orbits in the semisimple Lie algebra $\mathfrak{g}'_{\bar{0}}$.

Let $m,n\in\mathbb{Z}_{+}$.  We say that $(p,q)$ is a partition of $(m|n)$ if $p$ is a partition of $m$ and $q$ is a partition of $n$. There is a one-to-one correspondence between $G$-orbits of nilpotent even elements in $\mathfrak{gl}(m|n)$ and partitions of $(m|n)$. Two nilpotent even elements of $\mathfrak{osp}(m|2n)$ belong to the same $O(m)\times SP(2n)$ orbit if and only if they belong the same $GL(m)\times GL(2n)$ orbit.  This follows from the theory of nilpotent orbits in finite-dimensional simple Lie algebras (see for example \cite{J}).

Given a partition  $p$, we let $p_1>\cdots >p_{a}$ be the distinct nonzero parts of $p$, and
we write $p=(p_1^{m_{p_1}},\ldots,p_{a}^{m_{p_{a}}})$, where $m_{p_i}$ is the multiplicity of $p_i$ in $p$.  A partition is called {\em  symplectic} (resp. {\em orthogonal}) if $m_{p_i}$ is even for odd $p_i$ (resp. even $p_i$).
We say that a partition $(p|q)$ of $(m|2n)$ is orthosymplectic if $p$ is an orthogonal partition of $m$ and $q$ is a symplectic partition of $2n$.
There is a one-to-one correspondence between $G$-orbits of nilpotent even elements in $\mathfrak{osp}(m|2n)$ and orthosymplectic partitions of $(m|n)$.  See Section~\ref{section7} for a description of orbit representatives.

\subsection{Centralizers of nilpotent even elements}

In this section, we describe the centralizer $\mathfrak{g}^{e}$ of a nilpotent even element $e\in\mathfrak{g}$ by choosing a nice basis of $V_{0}\oplus V_{1}$ (and hence of $\mbox{End}(V_{0}\oplus V_{1})$), which we use to compute the dimensions of $\mathfrak{g}_{\bar{0}}^{e}$ and $\mathfrak{g}_{\bar{1}}^{e}$.  This is analogous to the Lie algebra case \cite{J}. This was done for $\mathfrak{gl}(m|n)$ in \cite{WZ} for a field of prime characteristic, but the construction is identical in characteristic zero.

\subsubsection{Centralizers of nilpotent even elements in $\mathfrak{gl}(m|n)$}\label{secbas}
Let $\mathfrak{g}=\mathfrak{gl}(m|n):=\mbox{End}(V_{0}\oplus V_{1})$, where $\mathfrak{g}_{\bar{0}}=\mbox{End}(V_{0}) \oplus \mbox{End}(V_{1})$,  $\mathfrak{g}_{\bar{1}}=\mbox{Hom}(V_{0},V_{1}) \oplus \mbox{Hom}(V_{1},V_{0})$ and $\mbox{dim }V_{0}=m$, $\mbox{dim }V_{1}=n$.
Let $e\in\mathfrak{g}$ be a nilpotent element such that $e\in\mathfrak{g}_{\bar{0}}$. Then $e$ corresponds to some partition $(p,q)$ of $(m|n)$ given by positive integers
 $p_1\geq  \cdots \geq p_r$, $q_1\geq \cdots\geq q_s$, respectively.

Since $e$ is a nilpotent element in $\mbox{End}(V_{0}) \oplus \mbox{End}(V_{1})$, there exist $v_1,\ldots,v_r \in V_0$, $w_1,\ldots,w_s \in V_1$ such that $\{e^{j}v_i \mid 1\leq i\leq r,\ 0\leq j< p_i\}$ is a basis for $V_{0}$ and $\{e^{j}w_i \mid 1\leq i \leq s,\ 0\leq j< q_i\}$ is a basis for $V_{1}$, and $e^{p_i}v_i=0$, $e^{q_i}w_i=0$ by \cite{J}. Each element $Z\in \mathfrak{g}^{e}$ is determined by $Z(v_i), 1\leq i\leq r$ and $Z(w_i), 1\leq i \leq s$, since $Z(e^j v_i)=e^{j} Z(v_i)$ and $Z(e^j w_i)=e^{j} Z(w_i)$.

For $Z\in\mathfrak{g}^{e}_{\bar{0}}$ one has
\begin{equation}\label{evenz}
Z(v_i)=\sum_{j=1}^{r} \sum_{k=max\{0,p_j-p_i\}}^{p_j-1} \alpha_{k,j:i} e^{k} v_j, \hspace{.5cm}
Z(w_i)=\sum_{j=1}^{s} \sum_{k=max\{0,q_j-q_i\}}^{q_j-1} \beta_{k,j:i} e^{k} w_j
\end{equation}

For $Z\in\mathfrak{g}^{x}_{\bar{1}}$ one has
\begin{equation}\label{oddz}
Z(v_i)=\sum_{j=1}^{s} \sum_{k=max\{0,q_j-p_i\}}^{q_j-1} \gamma_{k,j:i} e^{k} w_j, \hspace{.5cm}
Z(w_i)=\sum_{j=1}^{r} \sum_{k=max\{0,p_i-q_j\}}^{p_i-1} \delta_{k,j:i} e^{k} v_j
\end{equation}

Since the coefficients $\alpha_{k,j:i}, \beta_{k,j:i}, \gamma_{k,j:i}, \delta_{k,j:i}$  can be chosen arbitrarily, the dimensions of $\mathfrak{g}^{e}_{\bar{0}}$ and $\mathfrak{g}^{e}_{\bar{1}}$ are determined by the number of indices. Hence by \cite{J, WZ}, we have
\begin{align*}\mbox{dim }\mathfrak{g}^{e}_{\bar{0}} &=\sum_{i,j=1}^{r} \mbox{min}(p_i,p_j) + \sum_{i,j=1}^{s} \mbox{min}(q_i,q_j)\\ &=\left(m+2\sum_{i=1}^{r}(i-1)p_i\right) +\left(n+2\sum_{j=1}^{s}(j-1)q_j\right),\\
\mbox{dim }\mathfrak{g}^{e}_{\bar{1}} &=2 \sum_{i,j=1}^{r,s} \mbox{min}(p_i,q_j).\end{align*}

The following lemma will be used in the proof of Theorem~\ref{thmgl}.

\begin{lemma}\label{glproof}
Let $\mathfrak{a}=\mathfrak{gl}(m)\times\mathfrak{gl}(n)$ and let
$i:\mathfrak{a}\hookrightarrow\mathfrak{gl}(m|n)$, $j:\mathfrak{a}\hookrightarrow\mathfrak{gl}(m+n)$ be the natural inclusion maps. Then $\mathfrak{gl}(m|n)$ and $\mathfrak{gl}(m+n)$ are isomorphic as $\mathfrak{a}$-modules under the adjoint action. Hence $\mbox{dim }\mathfrak{gl}(m|n)^{i(x)} = \mbox{dim }\mathfrak{gl}(m+n)^{j(x)}$ for all $x\in \mathfrak{a}$.
\end{lemma}

For $m\in\mathbb{N}$, let $Par(m)$ denote the set of partitions of $m$. Then for $m,n\in\mathbb{N}$ we have a natural map $\psi_{m,n}: Par(m)\times Par(n)\rightarrow Par(m+n)$.

\begin{lemma}If $e$ is a nilpotent element corresponding to the partition $(p,q)$ of $(m|n)$ and $\psi_{m,n}(p,q)=r=(r_1,\ldots,r_k)\in Par(m+n)$, then $\mbox{dim }\mathfrak{gl}(m|n)^{e}=\mathfrak{gl}(m+n)^{e}= (r_1^{*})^{2} +\cdots (r_N^{*})^{2}$, where $(r_1^{*},\ldots,r_N^{*})$ is the dual partition.
\end{lemma}

\subsubsection{Centralizers of nilpotent even elements in $\mathfrak{osp}(m|2n)$}\label{sectionvarphi}
Now $\mathfrak{g}=\mathfrak{osp}(m|2n)\subset\mathfrak{gl}(m|2n)=\mbox{End}(V_{0}\oplus V_{1})$ is defined as follows.
Let $\varphi$ be a non-degenerate supersymmetric bilinear form on $V=V_{0}\oplus V_{1}$, so that $V_{0}$ and $V_{1}$ are orthogonal and the restriction to $V_0$ is symmetric while the restriction to $V_{1}$ is skew-symmetric.  Then  $\varphi(x,y) = (-1)^{\mathrm{deg}x\ \mathrm{deg}y} \varphi(y,x)$ for all homogeneous elements $x,y\in V$. Define
$$\mathfrak{osp}(m|2n)_{i}:=\{z\in\mathfrak{gl}(m|2n)_{i} \mid \varphi(z(x),y)= -(-1)^{i (\mathrm{deg}\ x)}\varphi(x,z(y))\}.$$

Let $e\in\mathfrak{g}$ be a nilpotent element in $\mathfrak{g}$ such that $e\in\mathfrak{g}_{\bar{0}}$.
Then $e$ corresponds to some orthosymplectic partition $(p,q)$ of $(m|n)$ given by  $p_1\geq p_2\geq \cdots \geq p_r >0$, $q_1\geq q_2\geq\cdots\geq q_s >0$.
Since $e \in \mbox{End}(V_{0}) \oplus \mbox{End}(V_{1})$, by \cite{J} there exist $v_1,\ldots,v_r \in V_0$, $w_1,\ldots,w_s \in V_1$ such that $\{e^{j}v_i \mid 1\leq i\leq r,\ 0\leq j< p_i\}$ is a basis for $V_{0}$ and $\{e^{j}w_i \mid 1\leq i \leq s,\ 0\leq j< q_i\}$ is a basis for $V_{1}$, $e^{p_i}v_i=0$, $e^{q_i}w_i=0$, and which satisfy the following:
\begin{itemize}
  \item for each odd $p_i$,
$$
\varphi(e^{j}v_i,e^{h}v_k)=\left\{
                             \begin{array}{ll}
                               (-1)^j, & \hbox{if $k=i$ and $j+h=p_i-1$;} \\
                               0, & \hbox{otherwise;}
                             \end{array}
                           \right.
$$
\item for each even $p_i$, there exists  $\mu_{i}\in\{0,1\}$ and index $i^{*}\neq i$, $1\leq i^{*}\leq r$, such that $p_{i^*}=p_i$ and
$$
\varphi(e^{j}v_i,e^{h}v_{k})=\left\{
                             \begin{array}{ll}
                               (-1)^j\mu_i, & \hbox{if $k=i^*$ and $j+h=p_i-1$;} \\
                               0, & \hbox{otherwise;}
                             \end{array}
                           \right.
$$
\item for each odd $q_i$, there exists  $\omega_{i}\in\{0,1\}$ and index  $i^{*}\neq i$, $1\leq i^{*}\leq s$, such that $q_{i^*}=q_i$ and
$$
\varphi(e^{j}w_i,e^{h}w_{k})=\left\{
                             \begin{array}{ll}
                               (-1)^j\omega_i, & \hbox{if $k=i^*$ and $j+h=q_i-1$;} \\
                               0, & \hbox{otherwise;}
                             \end{array}
                           \right.
$$
\item for each even $q_i$
$$
\varphi(e^{j}w_i,e^{h}w_k)=\left\{
                             \begin{array}{ll}
                               (-1)^j, & \hbox{if $k=i$ and $j+h=q_i-1$;} \\
                               0, & \hbox{otherwise.}
                             \end{array}
                           \right.
$$
\end{itemize}
Note that $\varphi(e^{j}v_i,e^{h}w_k)=0$ for all $h,i,j,k$.

Now $\mathfrak{g}^{e}=\mathfrak{g}\cap \mathfrak{gl}(m|2n)^{e}$, so let $Z\in\mathfrak{gl}(m|2n)^{e}$.
For $Z\in\mathfrak{g}^{e}_{\bar{0}}=\mathfrak{so}(m)^{e}\times\mathfrak{sp}(2n)^{e}$, the  coefficients $\alpha_{k,j:i}, \beta_{k,j:i}$ of $Z(v_i)$ and $Z(w_i)$ in (\ref{evenz}) satisfy certain conditions given in \cite[Section 3.2]{J}, and since $\mbox{dim }\mathfrak{g}^{e}_{\bar{0}}=\mbox{dim}(\mathfrak{so}(m)^{e}) + \mbox{dim}(\mathfrak{sp}(2n)^{e})$,
we have that $\mbox{dim }\mathfrak{g}^{e}_{\bar{0}}=$
$$
\left(\frac{m}{2}+ \sum_{i=1}^{r}(i-1)p_i -\frac{1}{2}|\{i\mid p_i\text{ odd}\}|\right)+
\left(\frac{n}{2}+ \sum_{j=1}^{s}(j-1)q_j+\frac{1}{2}|\{j\mid q_j\text{ odd}\}|\right).
$$
For $Z\in\mathfrak{g}^{e}_{\bar{1}}$ , the coefficients $\gamma_{k,j:i}$ of $Z(v_i)$ in  (\ref{oddz}) can be chosen freely, but then the coefficients $\delta_{k,j:i}$ of $Z(w_i)$ are completely determined from this choice.
So the dimension of $\mathfrak{osp}(m|2n)_{\bar{1}}^{e}$ is one-half the dimension of  $\mathfrak{gl}(m|2n)_{\bar{1}}^{e}$. Hence,
$$\mbox{dim }\mathfrak{g}^{e}_{\bar{1}} = \sum_{i,j=1}^{r,s} \mbox{min}(p_i,q_j).$$

\subsection{Centralizers of $\mathfrak{sl}_2$-triples}
Fix an $\mathfrak{sl}_2$-triple $\mathfrak{s}=\{e,f,h\}\subset\mathfrak{g}_{\bar{0}}'$ satisfying $[e,f]=h$, $[h,e]=2e$ and $[h,f]=-2f$.  It is uniquely determined up to conjugacy by the nilpotent element $e$ \cite{Kos}. Let
$\mathfrak{g}=\oplus_{j\in\mathbb{Z}}\mathfrak{g}(j)$ be the $\mathbb{Z}$-grading given by the eigenspaces of $\mbox{ad }h$. Then $\mathfrak{g}^{e}=\oplus_{j\geq 0}\mathfrak{g}^{e}(j)$ and $\mathfrak{g}^{\mathfrak{s}}=\mathfrak{g}^{e}(0)$.
The following lemma can be proven using the same argument as in \cite{J}.

\begin{lemma}
$\mathfrak{g}^{e}$ is the semidirect product of the subalgebra $\mathfrak{g}^{e}(0)=\mathfrak{g}^{\mathfrak{s}}$ and the ideal  $\oplus_{m> 0}\mathfrak{g}^{e}(m)$. This ideal consists of nilpotent elements.
\end{lemma}

Let $(p|q)$ be a partition of $(m|n)$.  Let $r\in Par(m+n)$ be the total ordering of the partitions $p$ and $q$. Let $r_1>\cdots >r_{N}$ be the set of distinct nonzero parts of $r$.  Write $r=(r_1^{m_{1}+n_{1}},\ldots,r_{N}^{m_{N}+n_N})$ where $p=(r_1^{m_{1}},\ldots,r_{N}^{m_N})$ and $q=(r_1^{n_{1}},\ldots,r_{N}^{n_N})$. Note that for each $i$, one of $m_{i}$, $n_{i}$ can be zero.

For each $t\in\mathbb{Z}_{+}$, let $M_t=\sum_{i:p_i=t}\mathbb{F}v_i+\sum_{i:q_i=t}\mathbb{F}w_i$, where $v_i$, $w_i$ are as given in Section~\ref{secbas}.

\begin{theorem}\label{thmgl} Let $\mathfrak{g}=\mathfrak{gl}(m|n)$.
Let $e$ be a nilpotent even element corresponding to a partition $(p,q)$ of $(m|n)$, and let $\mathfrak{s}=\{e,f,h\}\subset\mathfrak{g}_{\bar{0}}'$ be an $\mathfrak{sl}_2$-triple for $e$. Then we have an isomorphism
$\mathfrak{g}^{\mathfrak{s}}\iso \mathfrak{gl}(m_1,n_1)\times  \cdots\times \mathfrak{gl}(m_N,n_N)$ of Lie superalgebras.
\end{theorem}

\begin{proof}
If $Z\in\mathfrak{g}^{\mathfrak{s}}\subset\mathfrak{g}^{e}$, then a coefficient of $Z(v_i)$, $Z(w_i)$ in (\ref{evenz}), (\ref{oddz}) is zero unless $k=0$ and $p_i,q_i=p_j,q_j$. So $Z(M_t)\subset M_t$. Since $Z$ is determined by all the $Z(v_i)$, $Z(w_i)$, the map
\begin{equation}\label{themap}\mathfrak{g}^{e}(0)\rightarrow \mathfrak{gl}(M_1)\times\mathfrak{gl}(M_2)\times\mathfrak{gl}(M_3)\times\cdots\end{equation} defined by restriction is injective.  Since these coefficients can be chosen freely, this map is surjective.  The even (resp. odd) dimension of $M_t$ is the number of $p_i$ (resp. $q_i$) with $p_i=t$ (resp. $q_i=t$).
\end{proof}

Let $(p|q)$ be an orthosymplectic partition of $(m|2n)$.  Let $r$ (resp. $s$) be the total ordering of the odd parts (resp. even parts) of the partitions $p$ and $q$.  Write $r=(r_1^{m_{1}+2n_{1}},\ldots,r_{N}^{m_{N}+2n_N})$ and $s=(s_1^{2c_{1}+d_{1}},\ldots,s_{T}^{2c_{T}+d_T})$ where $p=(r_1^{m_{1}},\ldots,r_{N}^{m_N},s_1^{2c_{1}},\ldots,s_{T}^{2c_{T}})$ and $q=(r_1^{2n_{1}},\ldots,r_{N}^{2n_N},s_1^{d_{1}},\ldots,s_{T}^{d_T})$.

\begin{theorem}Let $\mathfrak{g}=\mathfrak{osp}(m|2n)$.
Let $e$ be a nilpotent even element corresponding to an orthosymplectic partition $(p,q)$ of $(m|n)$, and let $\mathfrak{s}=\{e,f,h\}\subset\mathfrak{g}_{\bar{0}}'$ be an $\mathfrak{sl}_2$-triple for $e$. Then we have an isomorphism
$$\mathfrak{g}^{\mathfrak{s}}\iso \mathfrak{osp}(m_1,2n_1)\times \cdots  \times \mathfrak{osp}(m_N,2n_N)\times\mathfrak{osp}(d_1,2c_1)\times   \cdots  \times \mathfrak{osp}(d_T,2c_T)$$
of Lie superalgebras.
\end{theorem}

\begin{proof}
For each $t\in\mathbb{Z}_{+}$ define a bilinear form on $M_t$ by $\psi_t(x,y)=\varphi(x,e^{t-1}y)$.  Then $\psi_t$ is nondegenerate since for $v_i,v_k,w_i,w_k\in M_t$, we have $\psi_t(v_i,v_{k})=\delta_{k,i^{*}}\mu_i$, $\psi_t(w_i,w_{k})=\delta_{k,i^{*}}\omega_i$.   Moreover, $\psi_t$ is supersymmetric if $t$ is odd, and skew-supersymmetric if $t$ is even. Indeed, for homogeneous elements $x,y\in M_t$ we have
\begin{align*}
\psi_t(x,y)&=\varphi(x,e^{t-1}y)=(-1)^{(\mathrm{deg}\ x)(\mathrm{deg}\ y)} \varphi(e^{t-1}y,x)\\ &=(-1)^{(\mathrm{deg}\ x) (\mathrm{deg}\ y) +(t-1)}\varphi(y,e^{t-1}x)\\ &=(-1)^{(\mathrm{deg}\ x)(\mathrm{deg}\ y) +(t-1)}\psi_t(y,x).
\end{align*}
It is clear that $(M_t)_{0}$ is orthogonal to $(M_t)_1$ with respect to $\psi_t$ for all $t\in\mathbb{Z}_{+}$, since $e\in\mathfrak{g}_{\bar{0}}$.  Let $\Pi(N)$ be the superspace isomorphic to $N$ with switched parity. Then for each $t\in 2\mathbb{Z}_{+}$, the bilinear form $\psi_t:\Pi(M_t)\times\Pi(M_t)\rightarrow\mathbb{C}$ is supersymmetric.

If $Z\in\mathfrak{g}^{e}(0)_i$, then $Z(M_t)\subset M_t$, and for homogeneous $x,y\in M_t$ we have
\begin{align*}
\psi_t(Z(x),y)&=\varphi(Z(x),e^{t-1}y)=-(-1)^{i(\mathrm{deg}\ x)}\varphi(x,Z(e^{t-1}y))\\
&=-(-1)^{i(\mathrm{deg}\ x)}\varphi(x,e^{t-1}Z(y))\\ &=-(-1)^{i(\mathrm{deg}\ x)}\psi_t(x,Z(y))
\end{align*}
Hence, the homomorphism in (\ref{themap}) defines an injective map
\begin{equation*}\label{themap}\mathfrak{g}^{e}(0)\rightarrow \mathfrak{osp}(M_1)\times\mathfrak{osp}(\Pi(M_2))\times\mathfrak{osp}(M_3)\times\mathfrak{osp}(\Pi(M_4))
\times\mathfrak{osp}(M_5)\times\cdots.\end{equation*}
The fact that this map is surjective can be checked by direct computation. In particular, one should check that if $Z\in\mathfrak{gl}(m|2n)^{e}(0)_i$ satisfies $\psi_t(Z(x),y)=-(-1)^{i(\mathrm{deg}\ x)}\psi_t(x,Z(y))$ for all homogeneous $x,y\in M_t$ and for all $t\in\mathbb{Z}_{+}$, then $Z\in\mathfrak{osp}(m|2n)$.
\end{proof}

\section{Good $\mathbb{Z}$-gradings}\label{section4}

\subsection{Good $\mathbb{Z}$-gradings of Lie superalgebras}
\begin{defn}\label{gooddefinition}
Let $ \mathfrak{g}=\oplus_{j\in\mathbb{Z}} \mathfrak{g}(j)$  be a $\mathbb{Z}$-graded Lie superalgebra.  An element $e\in\mathfrak{g}_{\bar{0}}(2)$ is called {\em good} if the following properties hold:
\begin{equation}\label{good1} \mbox{ad }e: \mathfrak{g}(j)\rightarrow \mathfrak{g}(j+2) \text{ is injective for } j\leq -1; \end{equation}
\begin{equation}\label{good2} \mbox{ad }e: \mathfrak{g}(j)\rightarrow \mathfrak{g}(j+2) \text{ is surjective for } j\geq -1. \end{equation}
A $\mathbb{Z}$-grading of a Lie superalgebra $\mathfrak{g}$ is called {\em good} if it admits a good element.
\end{defn}

Clearly, (\ref{good1}) is equivalent to
\begin{equation}\label{good3} Ker(\mbox{ad }e) = \mathfrak{g}^{e} \subset \mathfrak{g}_{\geq},\end{equation}
and (\ref{good2}) is equivalent to
\begin{equation}\label{good4} \mathfrak{g}_{+} \subset Im(\mbox{ad }e) = [e,\mathfrak{g}].\end{equation}

\begin{lemma}\label{equiv}
Let $\mathfrak{g}$ be a $\mathbb{Z}$-graded Lie superalgebra. If $\mathfrak{g}$ is a basic Lie superalgebra or a semisimple Lie algebra, or if $\mathfrak{g}=\mathfrak{gl}(m|n)$ or $\mathfrak{sl}(n|n)$ and $\mathfrak{Z}(\mathfrak{g}_{\bar{0}})\subset\mathfrak{g}(0)$, then for $e\in\mathfrak{g}_{\bar{0}}(2)$ conditions (\ref{good1})-(\ref{good4}) are equivalent.
\end{lemma}

\begin{proof} If  $\mathfrak{g}=\mathfrak{gl}(m|n)$, we may restrict to $\mathfrak{sl}(m|n)$ since $\mathfrak{Z}(\mathfrak{g}_{\bar{0}})\subset\mathfrak{g}(0)$.
By Lemma~\ref{pairing}, the proof of \cite[Theorem 1.3]{EK} proves (\ref{good1}) $\Leftrightarrow$ (\ref{good2})
 \end{proof}

\begin{lemma}\label{evengood} Let $\mathfrak{g}$ be a $\mathbb{Z}$-graded Lie superalgebra.
If  $ \mathfrak{g}=\oplus_{j\in\mathbb{Z}} \mathfrak{g}(j)$ is a good grading for an element $e\in\mathfrak{g}_{\bar{0}}(2)$, then the induced grading of $ \mathfrak{g}_{\bar{0}}$ is a good grading for $e$.  Moreover, if $\mathfrak{g}$ is a basic Lie superalgebra, then $\mathfrak{Z}(\mathfrak{g})\subset\mathfrak{g}(0)$, $\mathfrak{Z}(\mathfrak{g}_{\bar{0}})\subset\mathfrak{g}_{\bar{0}}(0)$ and $e\in\mathfrak{g}'_{\bar{0}}(2)=\mathfrak{g}_{\bar{0}}(2)$.
\end{lemma}
\begin{proof}
Now  $\mbox{ad }e$ preserves parity since $e\in\mathfrak{g}_{\bar{0}}$, i.e. $(\mbox{ad }e)(\mathfrak{g}_{\bar{0}})\subset \mathfrak{g}_{\bar{0}} $, $(\mbox{ad }e)(\mathfrak{g}_{\bar{1}})\subset \mathfrak{g}_{\bar{1}} $.  So the map
$\mbox{ad }e: \mathfrak{g}(j)\rightarrow \mathfrak{g}(j+2)$
 is surjective (resp. injective) if and only if the maps
$\mbox{ad }e: \mathfrak{g}_{\bar{0}}(j)\rightarrow \mathfrak{g}_{\bar{0}}(j+2)$,
 $\mbox{ad }e: \mathfrak{g}_{\bar{1}}(j)\rightarrow \mathfrak{g}_{\bar{1}}(j+2) $
are both surjective (resp. injective).  In particular, if the $\mathbb{Z}$-grading of $\mathfrak{g}$ is a good grading for $e$, then the induced grading of $\mathfrak{g}_{\bar{0}}$ is a good grading for $e$.
The second claim now follows from Lemma~\ref{equiv} since $\mathfrak{Z}(\mathfrak{g})\subset \mathfrak{g}^{e}\subset \mathfrak{g}_{\geq 0}$ and
 $(\mathfrak{Z}(\mathfrak{g})\cap \mathfrak{g}_{+}) \subset (\mathfrak{Z}(\mathfrak{g})\cap Im(\mbox{ad }e)) = 0$.
\end{proof}

The proofs of the following lemmas are straightforward.

\begin{lemma}\label{reductive}
Let $\mathfrak{a}=\mathfrak{a}'\times\mathfrak{Z}(\mathfrak{a})$ be a reductive Lie algebra. Then  $ \mathfrak{a}=\oplus_{j\in\mathbb{Z}} \mathfrak{a}(j)$ is good $\mathbb{Z}$-grading for $e\in\mathfrak{a}(2)$ if and only if $\mathfrak{Z}(\mathfrak{a})\subset\mathfrak{a}(0)$, $e\in\mathfrak{a}'(2)$ and $\mathfrak{a}'=\oplus_{j\in\mathbb{Z}} \mathfrak{a}'(j)$ is a good $\mathbb{Z}$-grading for $e$.
\end{lemma}

\begin{lemma}\label{semisimple}
Let $\mathfrak{a}$ be a semisimple Lie algebra, and let $\mathfrak{a}=I_1 \oplus\cdots\oplus I_k$ be the unique decomposition of $\mathfrak{a}$ into ideals such that $I_i$ are simple as Lie algebras. Let $e\in\mathfrak{a}(2)$, and write $e= e_1 + \cdots + e_k$ where $e_i\in I_i$.
Then $\mathfrak{a}=\oplus_{j\in\mathbb{Z}} \mathfrak{a}(j)$ is a good (Dynkin) grading for $e$ if and only if for each $i$ the induced grading $I_i=\oplus_{j\in\mathbb{Z}} I_i(j)$ is a good (Dynkin) grading for the element $e_i$.
\end{lemma}

\subsection{Good $\mathbb{Z}$-gradings of basic Lie superalgebras}\label{sectionbasic}

Let $\mathfrak{g}$ be a basic Lie superalgebra,  $\mathfrak{gl}(m|n)$ or $\mathfrak{sl}(n|n)$. Then $\mathfrak{g}_{\bar{0}}=\mathfrak{g}_{\bar{0}}'\oplus\mathfrak{Z}(\mathfrak{g}_{\bar{0}})$ is a reductive Lie algebra.
A $\mathbb{Z}$-grading is called a {\em Dynkin grading} if it is defined by $\mbox{ad }h$, where $h$ belongs to an $\mathfrak{sl}_{2}$-triple $\mathfrak{s}=\{e,f,h\}$ satisfying $[e,f]=h$, $[h,e]=2e$ and $[h,f]=-2f$. By $\mathfrak{sl}_{2}$ theory, $e$ is a good element for this $\mathbb{Z}$-grading. Hence, all Dynkin gradings are good.  Moreover, every nilpotent even element has a good grading.
Indeed, we can apply the Jacobson-Morosov Theorem to $\mathfrak{g}_{\bar{0}}'$, since elements of $\mathfrak{Z}(\mathfrak{g}_{\bar{0}})$ are semisimple.

\begin{theorem}\label{commutes}
Let $\mathfrak{g}$ be a basic Lie superalgebra, $\mathfrak{gl}(m|n)$ or $\mathfrak{sl}(n|n)$.
Let $\mathfrak{g}=\oplus_{j\in\mathbb{Z}} \mathfrak{g}(j)$ be a good $\mathbb{Z}$-grading for $e\in\mathfrak{g}_{\bar{0}}(2)$, and let $\mathfrak{s}=\{e,f,h\}$ be an $\mathfrak{sl}_2$-triple containing $e$. Then $\mathfrak{g}^{\mathfrak{s}} \subset \mathfrak{g}(0)$ and  $\mathfrak{g}_{\bar{0}}^{\mathfrak{s}} \subset \mathfrak{g}_{\bar{0}}(0)$.
\end{theorem}

\begin{proof}By Lemma~\ref{evengood}, the induced grading of $\mathfrak{g}_{\bar{0}}$ is a good grading for $e$ and  $e \in \mathfrak{g}'_{\bar{0}}$.  Since $\mathfrak{g}'_{\bar{0}}$ is semisimple and $e$ is nilpotent, by the Jacobson-Morosov Theorem there exists an $\mathfrak{sl}_2$-triple $\mathfrak{s}=\{e,f,h\}\subset\mathfrak{g}'_{\bar{0}}$ containing $e$.
We have that $\mathfrak{g}^{\mathfrak{s}}=\oplus_{j\in\mathbb{Z}} \mathfrak{g}^{\mathfrak{s}}(j)$.
Now $\mathfrak{g}^{\mathfrak{s}}\subset\mathfrak{g}^{e}$ since $e\in\mathfrak{s}$ and $\mathfrak{g}^{e}\subset\mathfrak{g}_{\geq}$ by Lemma~\ref{equiv}, hence $\mathfrak{g}^{\mathfrak{s}}(j)=0$ for all $j\leq -1$.

Now $\mathfrak{g}$ is a finite dimensional module under the adjoint action of the $\mathfrak{sl}_2$ subalgebra generated by $\mathfrak{s}$.  By $\mathfrak{sl}_2$ representation theory,
\begin{equation*} \mathfrak{g} = \mathfrak{g}^{f} \oplus Im(\mbox{ad }e).  \eqno{(*)}\end{equation*} Since $\mathfrak{g}^{\mathfrak{s}} = \mathfrak{g}^{e}\cap \mathfrak{g}^{f}$, it follows from ($*$) that $\mathfrak{g}^{\mathfrak{s}}\cap Im(\mbox{ad }e)=\{0\}$. But then Lemma~\ref{equiv} implies that $\mathfrak{g}^{\mathfrak{s}}(j)=0$ for all $j\geq 1$. Therefore $\mathfrak{g}^{s} \subset \mathfrak{g}(0)$. The second claim follows from the fact that the induced grading of $\mathfrak{g}_{\bar{0}}$ is also a good grading for $e$.
\end{proof}

If $e\in\mathfrak{g}'_{\bar{0}}(2)$, $e\neq 0$, then  \cite[Lemma 1.1]{EK} gives the existence of $h\in\mathfrak{g}'_{\bar{0}}(0)$ and $f \in \mathfrak{g}'_{\bar{0}}(-2)$ such that
 $\mathfrak{s}=\{e,f,h\}$ is an $\mathfrak{sl}_2$-triple.

\begin{corollary}\label{cors}
Let $\mathfrak{g}$ be a basic Lie superalgebra, $\mathfrak{g}\neq \mathfrak{psl}(n|n)$, or let $\mathfrak{g}=\mathfrak{gl}(m|n)$. Let $\mathfrak{g}=\oplus_{j\in\mathbb{Z}} \mathfrak{g}(j)$ be a good $\mathbb{Z}$-grading for $e\in\mathfrak{g}_{\bar{0}}(2)$ defined by $H\in\mathfrak{g}_{\bar{0}}(0)$.  If $\mathfrak{s}=\{e,f,h\}$ is an $\mathfrak{sl}_2$-triple with $f \in \mathfrak{g}'_{\bar{0}}(-2)$ and $h\in\mathfrak{g}'_{\bar{0}}(0)$ given by \cite[Lemma 1.1]{EK}, then $z:=H-h\in \mathfrak{Z}(\mathfrak{g}^{\mathfrak{s}})_{\bar{0}}$.  In particular, if $\mathfrak{Z}(\mathfrak{g}^{\mathfrak{s}})_{\bar{0}}=\{0\}$, then the Dynkin grading is the only good grading for which $e$ is a good element.
\end{corollary}

We note that $H,h\in\mathfrak{g}_{\bar{0}}(0)$ are commuting semisimple elements, so we may choose our Cartan subalgebra to contain them. In particular, we see that all good gradings for an element $e$ can be described (up to equivalence) by adding a semisimple element $z\in\mathfrak{Z}(\mathfrak{g}^{\mathfrak{s}})_{\bar{0}}$ to the element $h$ of an $\mathfrak{sl}_2$-triple  $\mathfrak{s}=\{e,f,h\}$.\\

The proof of the following lemma is the same as for Lie algebras \cite{EK}.  Similarly, the theorem \cite[Theorem 1.4]{EK} and its corollaries can be extended to basic Lie superalgebras.

\begin{lemma}
Let $\mathfrak{g}$ be a basic Lie superalgebra and let $\Pi\subset\Delta_{\geq}$ be given by Lemma~\ref{positive}.  If $\mathfrak{g}=\oplus_{j\in\mathbb{Z}} \mathfrak{g}(j)$ is a good $\mathbb{Z}$-grading, then $\Pi=\Pi_{0}\sqcup\Pi_{1}\sqcup\Pi_{2}$ and $\Pi_{\bar{0}}=\Pi_{\bar{0},0}\sqcup\Pi_{\bar{0},1}\sqcup\Pi_{\bar{0},2}$.  In particular, all good gradings are among those defined by $\mbox{deg}(\alpha_i)=-\mbox{deg}(-\alpha_i)=0,1$, or $2,\ i=1,\ldots,n$ for some choice of simple roots $\Pi=\{\alpha_1,\ldots,\alpha_n\}$.
\end{lemma}

\subsection{Richardson elements}
Let $\mathfrak{g}$ be a basic Lie superalgebra. Let $\mathfrak{p}$ be a parabolic subalgebra of $\mathfrak{g}$, with nilradical $\mathfrak{n}$. We call an even or odd element $e\in\mathfrak{n}$ a {\em Richardson element} if $[\mathfrak{p},e]=\mathfrak{n}$.
For a finite-dimensional simple Lie algebra $\mathfrak{g}$ this definition is equivalent to the standard definition. If $G$ is the adjoint group of $\mathfrak{g}$, then an element $e$ in the nilradical $\mathfrak{n}$ is called a Richardson element for the Lie algebra $\mathfrak{p}$ of the parabolic subgroup $P\subset G$ if the orbit $Pe$ is open dense in $\mathfrak{n}$ \cite{C, EK}.

Recall that $\mathfrak{g}_{\geq}$ is a parabolic subalgebra of $\mathfrak{g}$ with nilradical $\mathfrak{g}_{+}$.

\begin{lemma}
Let $\mathfrak{g}=\oplus_{j\in 2\mathbb{Z}}\mathfrak{g}(j)$ be an even $\mathbb{Z}$-grading and let $\mathfrak{g}_{\geq}$ be the corresponding parabolic subalgebra of $\mathfrak{g}$. Let $e\in\mathfrak{g}_{\bar{0}}(2)$. Then $e$ is a Richardson element for $\mathfrak{g}_{\geq}$ if and only if $e$ is good.
\end{lemma}

\begin{proof}
Since  $e\in\mathfrak{g}(2)$ and the grading is even, $[\mathfrak{g}_{-},e]\subset\mathfrak{g}_{\leq}$. Clearly,  $[\mathfrak{g}_{\geq},e]\subset\mathfrak{g}_{+}$. By Lemma~\ref{equiv}, the grading is good for $e$ if and only if $\mathfrak{g}_{+}\subset [\mathfrak{g},e]$.  Hence, the grading is good for $e$ if and only if $[\mathfrak{g}_{\geq},e]=\mathfrak{g}_{+}$.
\end{proof}

It is important to note that a parabolic subalgebra of a basic Lie superalgebra does not necessarily have a Richardson element.  In particular, a Borel subalgebra of $\mathfrak{sl}(m|n)$ for $m\neq n\pm 1$ does not have a Richardson element.

\subsection{Extending good $\mathbb{Z}$-gradings of $\mathfrak{g}_{\bar{0}}$}
Given a basic Lie superalgebra $\mathfrak{g}$, it is natural to ask: which good $\mathbb{Z}$-gradings of $\mathfrak{g}_{\bar{0}}$ {\em extend} to good $\mathbb{Z}$-gradings of $\mathfrak{g}$, and to what extent is such an extension determined by the $\mathbb{Z}$-grading of $\mathfrak{g}_{\bar{0}}$.

The following lemma is easy to prove.

\begin{lemma}\label{induced}
Let $\mathfrak{g}$ be a Lie superalgebra.  If $H,H'\in\mathfrak{g}_{\bar{0}}$ are such that $\mbox{ad }H$ and $\mbox{ad }H'$ define $\mathbb{Z}$-gradings of $\mathfrak{g}$ for which the induced gradings of $\mathfrak{g}_{\bar{0}}$ coincide, then $H-H' \in\mathfrak{Z}(\mathfrak{g}_{\bar{0}})$.
\end{lemma}

The following lemma is a corollary of Lemma~\ref{induced}.

\begin{lemma}\label{extend}
Let $\mathfrak{g}$ be a basic Lie superalgebra such that $\mathfrak{Z}(\mathfrak{g}_{\bar{0}})=0$ and all derivations of $\mathfrak{g}$ are inner, i.e. $\mathfrak{osp}(m|2n) : m\neq 2n+2$, $F(4)$, $G(3)$, $D(2,1,\alpha)$. Then a $\mathbb{Z}$-grading of $\mathfrak{g}_{\bar{0}}$ has a unique extension to a $\mathbb{Z}$-grading of $\mathfrak{g}$.
 \end{lemma}

\begin{remark}
A good $\mathbb{Z}$-grading of $\mathfrak{g}_{\bar{0}}$ need not have a good extension to $\mathfrak{g}$. The main theorem will provide counterexamples. See Example~\ref{expyr2}.
\end{remark}

\begin{lemma}\label{Dynkin}
A Dynkin grading of $\mathfrak{g}_{\bar{0}}$ has an extension to a Dynkin grading of $\mathfrak{g}$.
\end{lemma}
\begin{proof}
Let $\mathfrak{g}_{\bar{0}}=\sum_{j\in\mathbb{Z}}\mathfrak{g}_{\bar{0}}(j)$ be a Dynkin grading defined by $\mbox{ad }h$ with good element $e\in\mathfrak{g}_{\bar{0}}(2)$ such that $\mathfrak{s}=\{e,f,h\}\subset\mathfrak{g}_{\bar{0}}$ is an $\mathfrak{sl}_{2}$-triple. Then $\mathfrak{g}$ is a finite dimensional module under the adjoint action of the $\mathfrak{sl}(2)$ subalgebra generated by $\mathfrak{s}$.  Hence it decomposes into a direct sum of irreducible modules.  The action $\mbox{ad }h$ defines a $\mathbb{Z}$-grading of $\mathfrak{g}$ for which $e$ is a good element.
\end{proof}

\begin{remark}
In the case that $\mathfrak{Z}(\mathfrak{g}_{\bar{0}})\neq 0$ it is possible that a good $\mathbb{Z}$-grading of $\mathfrak{g}_{\bar{0}}$ has more then one extension to a good $\mathbb{Z}$-grading of $\mathfrak{g}$.  See Example~\ref{expyr1}.
\end{remark}

\section{Good $\mathbb{Z}$-gradings for the exceptional basic Lie superalgebras}\label{section5}

\begin{theorem}\label{thmexc}
All good $\mathbb{Z}$-gradings of the exceptional Lie superalgebras $F(4)$, $G(3)$ and $D(2,1,\alpha)$ are Dynkin gradings.
\end{theorem}

\begin{proof}
Let $\mathfrak{g}$ be one of the exceptional basic Lie superalgebras, $F(4)$, $G(3)$ or $D(2,1,\alpha)$.  We see from Table 1 that the center of $\mathfrak{g}_{\bar{0}}$ is trivial. Let $\mathfrak{g}=\oplus_{j\in\mathbb{Z}}\mathfrak{g}(j)$ be a good $\mathbb{Z}$-grading with good element $e\in\mathfrak{g}(2)$.  The induced grading of each simple ideal of $\mathfrak{g}_{\bar{0}}$ is a good grading for $e$ by Lemma~\ref{evengood} and Lemma~\ref{semisimple}.  The grading of $\mathfrak{g}_{\bar{0}}$ is a Dynkin grading if and only if the induced grading of each simple ideal is a Dynkin grading.
All derivations of $\mathfrak{g}$ are inner \cite{K77, S}. Since $\mathfrak{Z}(\mathfrak{g}_{\bar{0}})=0$, an extension of a $\mathbb{Z}$-grading of $\mathfrak{g}_{\bar{0}}$ to a $\mathbb{Z}$-grading of $\mathfrak{g}$ is unique, by Lemma~\ref{extend}.  A Dynkin grading of $\mathfrak{g}_{\bar{0}}$ has an extension to a Dynkin grading of $\mathfrak{g}$.  Hence, if the induced grading of $\mathfrak{g}_{\bar{0}}$ is a Dynkin grading then the $\mathbb{Z}$-grading of $\mathfrak{g}$ is also Dynkin.
If $\mathfrak{g}=G(3)$ then $\mathfrak{g}_{\bar{0}}=G_2\times \mathfrak{sl}(2)$.  If $\mathfrak{g}=D(2,1,\alpha)$ then $\mathfrak{g}_{\bar{0}}=\mathfrak{sl}(2)\times \mathfrak{sl}(2) \times \mathfrak{sl}(2)$.
It was shown in \cite{EK} that every good $\mathbb{Z}$-grading of $G_{2}$ and of $\mathfrak{sl}(2)$ is a Dynkin grading. Hence, all good $\mathbb{Z}$-gradings of $G(3)$ and of $D(2,1,\alpha)$ are Dynkin gradings.

If $\mathfrak{g}=F(4)$, then $\mathfrak{g}_{\bar{0}}=\mathfrak{so}(7)\times \mathfrak{sl}(2)$.  By \cite{EK}, the only non-Dynkin gradings of $\mathfrak{so}(7)$ correspond to the nilpotent element with partition $(3,3,1)$.
The induced grading of $\mathfrak{sl}(2)$ is a good Dynkin grading.
By Lemma~$\ref{CSA}$, there exists a Cartan subalgebra $\mathfrak{h}\subset\mathfrak{g}_{\bar{0}}(0)$.  By Lemma~\ref{rootdecomposition}, the root space decomposition is compatible with the $\mathbb{Z}$-grading.  We fix the following set of simple roots for $F(4)$.

\xymatrixrowsep{.1pc}
$$\xymatrix{  \O  \AW[r]^{-1}_{-1} &  \O \AW[r]^{-1}_{-2}  &  \O \AW[r]^{-1}_{1}  &  \OX \\
\alpha_1 & \alpha_2 & \alpha_3 & \alpha_4}$$
\xymatrixrowsep{1.5pc}

Then $\{\alpha_1,\alpha_2,\alpha_3\}$ is a set of simple roots for the simple ideal isomorphic to $\mathfrak{so}(7)$.
The highest root $\theta=\alpha_1+2\alpha_2+3\alpha_3+2\alpha_4$ is a root for the simple ideal isomorphic to $\mathfrak{sl}(2)$.  This implies that $\mbox{Deg}(\theta) = \pm 2$.  The nilpotent element of $\mathfrak{so}(7)$ corresponding to the partition $(3,3,1)$ is (up to conjugacy) $e_1=X_1+X_2$ where $X_1\in\mathfrak{g}_{\alpha_1}$, $X_2\in\mathfrak{g}_{\alpha_2} $.  The Dynkin grading for $e_1$ is $[\mbox{deg}(\alpha_1),\mbox{deg}(\alpha_2),\mbox{deg}(\alpha_3)]$
=$[2,2,-2]$, and the non-Dynkin gradings are $[2,2,-1]$ and $[2,2,-3]$.  For the non-Dynkin gradings, we have that $\mbox{Deg}(\theta)=3+2\mbox{Deg}(\alpha_4)$ and $\mbox{Deg}(\theta)=-3+2\mbox{Deg}(\alpha_4)$, respectively.  Since $\mbox{Deg}(\alpha_4)\in\mathbb{Z}$ this implies $\mbox{Deg}(\theta)$ is odd, which is impossible since $\mbox{Deg}(\theta) =\pm 2$.
\end{proof}

\section{Good $\mathbb{Z}$-gradings for $\mathfrak{psl}(2|2)$}\label{section6a}

We adopt the notation of Lemma~\ref{rempsl}.  The following lemma can proven by explicitly computing $\mathfrak{g}^{e}$.
\begin{lemma} The $\mathbb{Z}$-grading of $\mathfrak{psl}(2|2)$ defined by $m\in\mathbb{Z}$, $p,q\in\{0,2\}$ and $(a:b), (c:d)\in\mathbb{P}^{2}$ satisfying $(a:b) \neq (c:d)$ is a good grading for the element $e=rE_{12}+sE_{34}$ if and only if $p=0 \Leftrightarrow r=0$, $q=0 \Leftrightarrow s=0$ and $0\leq m\leq p+q$.
\end{lemma}

\section{Good $\mathbb{Z}$-gradings for $\mathfrak{gl}(m|n)$}\label{section6}
In this section we classify the good $\mathbb{Z}$-gradings of $\mathfrak{g}=\mathfrak{gl}(m|n)$. The good $\mathbb{Z}$-gradings of $\mathfrak{sl}(m|n) : m\neq n$ and $\mathfrak{psl}(n|n): n\neq 2$ are uniquely induced from good $\mathbb{Z}$-gradings of $\mathfrak{g}=\mathfrak{gl}(m|n)$ since   $\mathfrak{Z}(\mathfrak{g}_{\bar{0}})\subset\mathfrak{g}(0)$.  See Lemma~\ref{indu}). To describe these gradings we generalize the definition of a pyramid given in \cite{BG, EK} to the Lie superalgebra $\mathfrak{gl}(m|n)$.

A pyramid $P$ is a finite collection of boxes of size $2\times 2$ in the upper half plane which are centered at integer coordinates, such that for each $j=1,\ldots,N$, the second coordinates of the $j^{th}$ row equal $2j-1$ and the first coordinates of the $j^{th}$ row form an arithmetic progression  $f_j,f_j+2,\ldots,l_j$ with difference $2$, such that the first row is centered at $(0,0)$, i.e.  $f_1=-l_1$, and
\begin{equation}\label{eqpyramid} f_{j}\leq f_{j+1}\leq l_{j+1}\leq l_{j} \hspace{.5cm} \text{for all } j.\end{equation}
Each box of $P$ has even or odd parity.  We say that $P$ has {\em size} $(m|n)$ if $P$ has exactly $m$ even boxes and $n$ odd boxes.

Fix $m,n\in\mathbb{Z}_{+}$ and let $(p,q)$ be a partition of $(m|n)$. Let $r=\psi(p,q)\in Par(m+n)$ be the total ordering of the partitions $p$ and $q$ which satisfies: if $p_i=q_j$ for some $i,j$ then $\psi(p_i)< \psi(q_j)$. We define $Pyr(p,q)$ to be the set of pyramids which satisfy the following two conditions:
(1) the $j^{th}$ row of a pyramid $P\in Pyr(p,q)$ has length $r_j$;
(2) if $\psi^{-1}(r_j)\in p$ (resp.  $\psi^{-1}(r_j)\in q$) then all boxes in the $j^{th}$ row have even (resp. odd parity) and we mark these boxes with a ``$+$''  (resp. ``$-$'' sign).

Corresponding to each pyramid $P\in Pyr(p,q)$ we define a nilpotent element $e(P)\in\mathfrak{g}_{\bar{0}}$ and semisimple element $h(P)\in\mathfrak{g}_{\bar{0}}$, as follows.  Recall $\mathfrak{gl}(m|n)=\mbox{End}(V_{0}\oplus V_{1})$.  Fix a basis $\{v_1,\ldots,v_m\}$ of $V_{0}$ and $\{v_{m+1},\ldots,v_{m+n}\}$ of $V_{1}$.
Label the even (resp. odd) boxes of $P$ by the basis vectors of $V_{0}$ (resp. $V_{1}$). Define an endomorphism $e(P)$ of $V_{0}\oplus V_{1}$ as acting along the rows of the pyramid, i.e. by sending a basis vector $v_i$ to the basis vector which labels the box to the right of the box labeled by $v_i$ or to zero if it has no right neighbor.   Then $e(P)$ is nilpotent and corresponds to the partition $(p,q)$.  Since $e(P)$ does not depend the choice of $P$ in $Pyr(p,q)$, we may denote it by $e_{p,q}$.  Moreover, $e_{p,q}\in\mathfrak{g}_{\bar{0}}$ because boxes in the same row have the same parity.

Define $h(P)$ to be the $(m+n)$-diagonal matrix where the $i^{th}$ diagonal entry is the first coordinate of the box labeled by the basis vector $v_{i}$.  Then $h(P)$ defines a $\mathbb{Z}$-grading of $\mathfrak{g}$ for which $e_{p,q}\in\mathfrak{g}(2)$.  Let $P_{p,q}$ denote the symmetric pyramid from $Pyr(p,q)$. Then $h(P_{p,q})$ defines a Dynkin grading for $e_{p,q}$, and $P_{p,q}$ is called the {\em Dynkin pyramid} for the partition $(p|q)$.

\begin{example}\label{expyr1}
Let $\mathfrak{g}=\mathfrak{gl}(4|6)$ and consider the partitions $p=(3,1)$ and $q=(4,2)$.  The Dynkin grading of $\mathfrak{g}_{\bar{0}}=\mathfrak{gl}(4)\times\mathfrak{gl}(6)$ for the partition $(p,q)$ corresponds to the following symmetric pyramids.

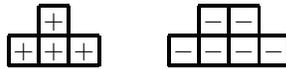
\setlength{\unitlength}{.4cm}
\begin{center}
\begin{picture}(5,2)(0,0)
\linethickness{1pt}
\put(0,0){\line(1,0){3}}
\put(0,1){\line(1,0){3}}
\put(1,2){\line(1,0){1}}
\put(0,0){\line(0,1){1}}
\put(1,0){\line(0,1){2}}
\put(2,0){\line(0,1){2}}
\put(3,0){\line(0,1){1}}
\put(0.15,0.25){$+$}
\put(1.15,0.25){$+$}
\put(2.15,0.25){$+$}
\put(1.15,1.25){$+$}
\end{picture}
\begin{picture}(6,2)(0,0)
\linethickness{1pt}
\put(0,0){\line(1,0){4}}
\put(0,1){\line(1,0){4}}
\put(1,2){\line(1,0){2}}
\put(0,0){\line(0,1){1}}
\put(1,0){\line(0,1){2}}
\put(2,0){\line(0,1){2}}
\put(3,0){\line(0,1){2}}
\put(4,0){\line(0,1){1}}
\put(0.15,0.25){$-$}
\put(1.15,0.25){$-$}
\put(2.15,0.25){$-$}
\put(3.15,0.25){$-$}
\put(1.15,1.25){$-$}
\put(2.15,1.25){$-$}
\end{picture}
\end{center}

There are pyramids in $Pyr(p,q)$ for which the induced grading of $\mathfrak{g}_{\bar{0}}$ is the one given above, and these correspond to good $\mathbb{Z}$-gradings. They are represented by the following pyramids:\\
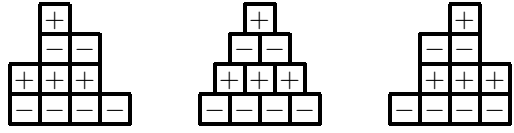
\begin{center}
\begin{picture}(6,4)(0,0)
\linethickness{1pt}
\put(0,0){\line(1,0){4}}
\put(0,1){\line(1,0){4}}
\put(0,2){\line(1,0){3}}
\put(1,3){\line(1,0){2}}
\put(1,4){\line(1,0){1}}
\put(0,0){\line(0,1){2}}
\put(1,0){\line(0,1){4}}
\put(2,0){\line(0,1){4}}
\put(3,0){\line(0,1){3}}
\put(4,0){\line(0,1){1}}
\put(0.15,1.25){$+$}
\put(1.15,1.25){$+$}
\put(1.15,3.25){$+$}
\put(2.15,1.25){$+$}
\put(0.15,0.25){$-$}
\put(1.15,0.25){$-$}
\put(2.15,0.25){$-$}
\put(3.15,0.25){$-$}
\put(1.15,2.25){$-$}
\put(2.15,2.25){$-$}
\end{picture}
\begin{picture}(6,4)(0,0)
\linethickness{1pt}
\put(0,0){\line(1,0){4}}
\put(0,1){\line(1,0){4}}
\put(.5,2){\line(1,0){3}}
\put(1,3){\line(1,0){2}}
\put(1.5,4){\line(1,0){1}}
\put(0,0){\line(0,1){1}}
\put(1,0){\line(0,1){1}}
\put(2,0){\line(0,1){1}}
\put(3,0){\line(0,1){1}}
\put(4,0){\line(0,1){1}}
\put(.5,1){\line(0,1){1}}
\put(1.5,1){\line(0,1){1}}
\put(2.5,1){\line(0,1){1}}
\put(3.5,1){\line(0,1){1}}
\put(1,2){\line(0,1){1}}
\put(2,2){\line(0,1){1}}
\put(3,2){\line(0,1){1}}
\put(1.5,3){\line(0,1){1}}
\put(2.5,3){\line(0,1){1}}
\put(0.65,1.25){$+$}
\put(1.65,1.25){$+$}
\put(2.65,1.25){$+$}
\put(1.65,3.25){$+$}
\put(0.15,0.25){$-$}
\put(1.15,0.25){$-$}
\put(2.15,0.25){$-$}
\put(3.15,0.25){$-$}
\put(1.15,2.25){$-$}
\put(2.15,2.25){$-$}
\end{picture}
\begin{picture}(6,4)(0,0)
\linethickness{1pt}
\put(0,0){\line(1,0){4}}
\put(0,1){\line(1,0){4}}
\put(1,2){\line(1,0){3}}
\put(1,3){\line(1,0){2}}
\put(2,4){\line(1,0){1}}
\put(0,0){\line(0,1){1}}
\put(1,0){\line(0,1){3}}
\put(2,0){\line(0,1){4}}
\put(3,0){\line(0,1){4}}
\put(4,0){\line(0,1){2}}
\put(1.15,1.25){$+$}
\put(2.15,1.25){$+$}
\put(3.15,1.25){$+$}
\put(2.15,3.25){$+$}
\put(0.15,0.25){$-$}
\put(1.15,0.25){$-$}
\put(2.15,0.25){$-$}
\put(3.15,0.25){$-$}
\put(1.15,2.25){$-$}
\put(2.15,2.25){$-$}
\end{picture}
\end{center}
\end{example}

\begin{theorem}\label{thmgl}  Let  $\mathfrak{g}=\mathfrak{gl}(m|n)$,
and let $(p,q)$ be a partition of $(m|n)$. If  $P$ is a pyramid from $Pyr(p,q)$, then the pair $(h(P),e_{p,q})$ is good. Moreover, every good grading for $e_{p,q}$ is of the form $(h(P),e_{p,q})$ for some pyramid $P\in Pyr(p,q)$.
\end{theorem}

\begin{proof}
By Lemma~\ref{glproof}, this can be proven using the same method as for $\mathfrak{gl}(m+n)$ given in \cite{EK}.  It is easy to see from \cite[Figures 1-3]{EK} that if $e=e_{p,q}$ and $P\in Pyr(p,q)$, then the eigenvalues of $\mbox{ad }h(P)$ on $\mathfrak{g}^{e}$ are nonnegative.   Conversely, a good $\mathbb{Z}$-grading for $e_{p,q}$ is defined by the eigenvalues of  $\mbox{ad}((h(P_{p,q})+z)$ where  $z\in\mathfrak{Z}(\mathfrak{g}^{\mathfrak{s}})_{\bar{0}}$ is a diagonal matrix with integer entries and $\mathfrak{s}=\{e_{p,q},h(P_{p,q}),f\}$ is an $\mathfrak{sl}_{2}$-triple (see Section~\ref{sectionbasic}).  It is easy to see from \cite[Figures 1-3]{EK} that the condition $z\in\mathfrak{Z}(\mathfrak{g}^{\mathfrak{s}})_{\bar{0}}$ implies that the diagonal entries of $z$ must be constant along each row of the pyramid and equal on rows of the same length. Moreover, condition (\ref{eqpyramid})  must be satisfied in order for the eigenvalues of $\mbox{ad}(h(P_{p,q})+z)$ on $\mathfrak{g}^{e}$ to be nonnegative. So $h(P_{p,q})+z=h(P)$ for some pyramid $P\in Pyr(p,q)$.
\end{proof}

\begin{example}\label{expyr2}
Let $\mathfrak{g}=\mathfrak{gl}(4|6)$ and consider the partitions $p=(3,1)$ and $q=(4,2)$.
The following pyramids represent a good $\mathbb{Z}$-grading of $\mathfrak{g}_{\bar{0}}$ for which
there is no good $\mathbb{Z}$-grading of $\mathfrak{g}$ with this induced good $\mathbb{Z}$-grading of $\mathfrak{g}_{\bar{0}}$.\\
\setlength{\unitlength}{.4cm}
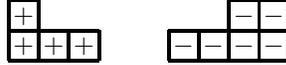
\begin{center}
\begin{picture}(5,2)(0,0)
\linethickness{1pt}
\put(0,0){\line(1,0){3}}
\put(0,1){\line(1,0){3}}
\put(0,2){\line(1,0){1}}
\put(0,0){\line(0,1){2}}
\put(1,0){\line(0,1){2}}
\put(2,0){\line(0,1){1}}
\put(3,0){\line(0,1){1}}
\put(0.15,0.25){$+$}
\put(1.15,0.25){$+$}
\put(2.15,0.25){$+$}
\put(0.15,1.25){$+$}
\end{picture}
\begin{picture}(6,2)(0,0)
\linethickness{1pt}
\put(0,0){\line(1,0){4}}
\put(0,1){\line(1,0){4}}
\put(2,2){\line(1,0){2}}
\put(0,0){\line(0,1){1}}
\put(1,0){\line(0,1){1}}
\put(2,0){\line(0,1){2}}
\put(3,0){\line(0,1){2}}
\put(4,0){\line(0,1){2}}
\put(0.15,0.25){$-$}
\put(1.15,0.25){$-$}
\put(2.15,0.25){$-$}
\put(3.15,0.25){$-$}
\put(2.15,1.25){$-$}
\put(3.15,1.25){$-$}
\end{picture}
\end{center}
\end{example}

\section{Good $\mathbb{Z}$-gradings for $\mathfrak{osp}(m|2n)$}\label{section7}
In this section we classify good $\mathbb{Z}$-gradings for $\mathfrak{g}=\mathfrak{osp}(m,2n)$. Recall that $\mathfrak{g}_{\bar{0}}=\mathfrak{so}(m)\times\mathfrak{sp}(2n)$. To describe these gradings we define an orthosymplectic pyramid, generalizing the definition of orthogonal and symplectic pyramids as defined in \cite{BG, EK}.

Given a partition  $p$, we let $J_p=\{p_1>\cdots >p_{a}\}$ be the set of distinct nonzero parts of $p$.
We write $p=(p_1^{m_{p_1}},\ldots,p_{a}^{m_{p_{a}}})$, where $m_{p_i}$ is the multiplicity of $p_i$ in $p$.  A partition is called orthogonal (resp. symplectic) if $m_{p_i}$ is even for even (resp. odd) $p_i$.
We say that a partition $(p|q)$ of $(m|2n)$ is orthosymplectic if $p$ is an orthogonal partition of $m$ and $q$ is a symplectic partition of $2n$.

Let $(p|q)$ be an orthosymplectic partition of $(m|2n)$.  Let $r\in Par(m+2n)$ be the total ordering of the partitions $p$ and $q$. Let $J_r=\{r_1>\cdots >r_{b}\}$ be the set of distinct nonzero parts of $r$.  Write $r=(r_1^{m_{1}+n_{1}},\ldots,r_{b}^{m_{b}+n_b})$ where $p=(r_1^{m_{1}},\ldots,r_{b}^{m_b})$ and $q=(r_1^{n_{1}},\ldots,r_{b}^{n_b})$.

We define the {\em orthosymplectic Dynkin pyramid} for $(p|q)$ as follows. It is a finite collection of boxes of size $2\times 2$ in the plane centered at integer coordinates: $(i,2j)$ for $m$ odd and $(i,2j-1)$ for $m$ even. It is centrally symmetric about $(0,0)$.  We describe how to place the boxes in the upper half plane. The boxes in lower half plane are obtained by the central symmetry.

If $m$ is even, then the zeroth row is empty.  If $m$ is odd, let $r_k$ be the largest part of $p$ occurring with odd multiplicity. Put $r_k$ boxes in the zeroth row in the columns $1-r_k,3-r_k,\ldots,r_k-1$, and remove one part of $r_k$ from the partition.  Now $p$ has an even number of parts occurring with odd multiplicity. Denote these by $c_1 >d_1 >\cdots > c_N > d_N$.

We add boxes inductively to the next row in the upper half plane as follows.  Let $r_j$ be the largest part remaining in the partition $r$.  If $m_j$ is odd, then $r_j=c_i$ for some $i$.  We add an ``even skew-row''  of length $\frac{c_i+d_i}{2}$ of even parity boxes in the columns $1-d_i,3-d_i,\ldots,c_i-1$, and then remove $c_i$ and $d_i$ from the partition.  Next we add $\lfloor\frac{m_j}{2}\rfloor$ rows of length $r_j$ of even parity boxes  in the columns $1-r_j,3-r_j,\ldots,r_j-1$. If $n_j$ is odd, we then add an ``odd skew-row'' of length $\frac{r_j}{2}$ of odd parity boxes  in the columns $1,\ldots,r_j-1$. Finally we add $\lfloor\frac{n_j}{2}\rfloor$ rows  of length $r_j$ of odd parity boxes in the columns $1-r_j,3-r_j,\ldots,r_j-1$, and remove $r_j^{m_{j}+n_{j}}$ from the partition.  We label the even boxes with the symbol ``$+$'' and the odd boxes with the symbol ``$-$''.

\begin{example}
$\mathfrak{osp}(9|6)$. The pyramids for the partitions $(5,3,1 | 3,3)$, $(4,4,1 | 6)$, $(7,1,1 |4,2 )$ are:\\
\setlength{\unitlength}{.4cm}
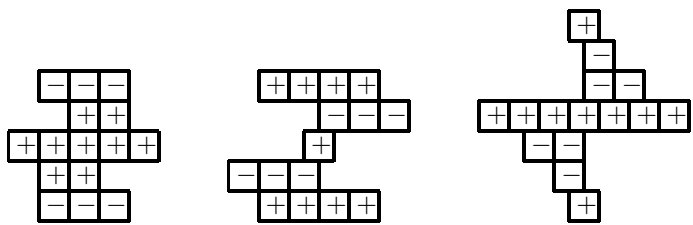
\begin{center}
\begin{picture}(7,7)(0,0)
\linethickness{1pt}
\put(1,0){\line(1,0){3}}
\put(1,1){\line(1,0){3}}
\put(0,2){\line(1,0){5}}
\put(0,3){\line(1,0){5}}
\put(1,4){\line(1,0){3}}
\put(1,5){\line(1,0){3}}
\put(0,2){\line(0,1){1}}
\put(1,0){\line(0,1){3}}
\put(1,4){\line(0,1){1}}
\put(2,0){\line(0,1){5}}
\put(3,0){\line(0,1){5}}
\put(4,0){\line(0,1){1}}
\put(4,2){\line(0,1){3}}
\put(5,2){\line(0,1){1}}
\put(1.25,0.3){$-$}
\put(2.25,0.3){$-$}
\put(3.25,0.3){$-$}
\put(1.25,1.3){$+$}
\put(2.25,1.3){$+$}
\put(0.25,2.3){$+$}
\put(1.25,2.3){$+$}
\put(2.25,2.3){$+$}
\put(3.25,2.3){$+$}
\put(4.25,2.3){$+$}
\put(2.25,3.3){$+$}
\put(3.25,3.3){$+$}
\put(1.25,4.3){$-$}
\put(2.25,4.3){$-$}
\put(3.25,4.3){$-$}
\end{picture}
\begin{picture}(8,7)(0,0)
\linethickness{1pt}
\put(1,0){\line(1,0){4}}
\put(0,1){\line(1,0){5}}
\put(0,2){\line(1,0){3.5}}
\put(2.5,3){\line(1,0){3.5}}
\put(1,4){\line(1,0){5}}
\put(1,5){\line(1,0){4}}
\put(0,1){\line(0,1){1}}
\put(1,0){\line(0,1){2}}
\put(1,4){\line(0,1){1}}
\put(2,0){\line(0,1){2}}
\put(2,4){\line(0,1){1}}
\put(2.5,2){\line(0,1){1}}
\put(3,0){\line(0,1){2}}
\put(3,3){\line(0,1){2}}
\put(3.5,2){\line(0,1){1}}
\put(4,0){\line(0,1){1}}
\put(4,3){\line(0,1){2}}
\put(5,0){\line(0,1){1}}
\put(5,3){\line(0,1){2}}
\put(6,3){\line(0,1){1}}
\put(1.25,0.3){$+$}
\put(2.25,0.3){$+$}
\put(3.25,0.3){$+$}
\put(4.25,0.3){$+$}
\put(0.25,1.3){$-$}
\put(1.25,1.3){$-$}
\put(2.25,1.3){$-$}
\put(2.75,2.3){$+$}
\put(3.25,3.3){$-$}
\put(4.25,3.3){$-$}
\put(5.25,3.3){$-$}
\put(1.25,4.3){$+$}
\put(2.25,4.3){$+$}
\put(3.25,4.3){$+$}
\put(4.25,4.3){$+$}
\end{picture}
\begin{picture}(10,7)(0,0)
\linethickness{1pt}
\put(3,0){\line(1,0){1}}
\put(2.5,1){\line(1,0){1.5}}
\put(1.5,2){\line(1,0){2}}
\put(0,3){\line(1,0){7}}
\put(0,4){\line(1,0){7}}
\put(3.5,5){\line(1,0){2}}
\put(3,6){\line(1,0){1.5}}
\put(3,7){\line(1,0){1}}
\put(0,3){\line(0,1){1}}
\put(1,3){\line(0,1){1}}
\put(1.5,2){\line(0,1){1}}
\put(2,3){\line(0,1){1}}
\put(2.5,1){\line(0,1){2}}
\put(3,0){\line(0,1){1}}
\put(3,3){\line(0,1){1}}
\put(3,6){\line(0,1){1}}
\put(3.5,1){\line(0,1){2}}
\put(3.5,4){\line(0,1){2}}
\put(4,0){\line(0,1){1}}
\put(4,3){\line(0,1){1}}
\put(4,6){\line(0,1){1}}
\put(4.5,4){\line(0,1){2}}
\put(5,3){\line(0,1){1}}
\put(5.5,4){\line(0,1){1}}
\put(6,3){\line(0,1){1}}
\put(7,3){\line(0,1){1}}
\put(3.25,0.3){$+$}
\put(2.75,1.3){$-$}
\put(1.75,2.3){$-$}
\put(2.75,2.3){$-$}
\put(0.25,3.3){$+$}
\put(1.25,3.3){$+$}
\put(2.25,3.3){$+$}
\put(3.25,3.3){$+$}
\put(4.25,3.3){$+$}
\put(5.25,3.3){$+$}
\put(6.25,3.3){$+$}
\put(3.75,4.3){$-$}
\put(4.75,4.3){$-$}
\put(3.75,5.3){$-$}
\put(3.25,6.3){$+$}
\end{picture}
\end{center}
\end{example}

We define a nilpotent element $e_{p,q}\in\mathfrak{g}_{\bar{0}}$ and semisimple element $h_{p,q}\in\mathfrak{g}_{\bar{0}}$ as in \cite{BG}. Let $\varphi$ be a non-degenerate supersymmetric bilinear form on $V=V_{0}\oplus V_{1}$, so that $V_{0}$ and $V_{1}$ are orthogonal and the restriction to $V_0$ is symmetric while the restriction to $V_{1}$ is skew-symmetric.
Let $k=\lfloor \frac{m}{2}\rfloor$. We take the standard basis $\{v_0,v_1,\ldots,v_k,v_{-1},\ldots,v_{-k}\}$ of $V_{\bar{0}}$ and $\{v_{k+1},\ldots,v_{k+n},v_{-(k+1)},\ldots,v_{-(k+n)}\}$ of $V_{\bar{1}}$, which for  $i,j>0$ satisfies  $\varphi(v_0,v_0)=2$, $\varphi(v_0,v_{\pm j})=0$, $\varphi(v_i,v_j)=\varphi(v_{-i},v_{-j})=0$, and $\varphi(v_i,v_{-j})=\delta_{ij}$.  We omit $v_0$ if $m=2k$.

We write $E_{i,j}$ for the matrix with a $1$ in the $(i,j)$ place and zeros elsewhere.
The following matrices give a Chevalley basis for $\mathfrak{osp}(m|2n)_{\bar{0}}=\mathfrak{so}(m)\times\mathfrak{sp}(2n)$ (omitting the first set if $m=2k$)
\begin{eqnarray*}
&\{2E_{i,0}-E_{0,-i}, E_{0,i}-2E_{-i,0}\}_{1\leq i\leq k} \cup \{E_{i,-j}-E_{j,-i}, E_{-j,i}-E_{-i,j}\}_{1\leq i<j\leq k} \\
&\cup \{E_{i,j}-E_{-j,-i}\}_{1\leq i,j\leq k} \cup \{E_{i,j}-E_{-j,-i}\}_{k+1 \leq i,j \leq k+n} \\
&\cup \{E_{i,-i},E_{-i,i}\}_{k+1\leq i\leq k+n}
\cup \{E_{i,-j}+E_{j,-i}, E_{-i,j}+E_{-j,i}\}_{k+1\leq i<j\leq k+n}.
\end{eqnarray*}
Define $\sigma_{i,j}\in\{\pm 1\}$ to be the coefficient of $E_{i,j}$ of the unique element in this basis if it appears, or zero if no basis element involves $E_{i,j}$ \cite{BG}.

Label the even boxes (resp. odd boxes) in the upper half plane $x,y>0$ with the vectors $v_1,\ldots,v_k$ (resp. $v_{k+1},\ldots,v_{k+n}$). The centrally symmetric box of the box labeled with $v_i$ is labeled with $v_{-i}$. There is a box at $(0,0)$ if and only if $m$ is odd, in which case we label this box with $v_0$.

Define $e_{p,q}$  to be the matrix $\sum_{i,j} \sigma_{i,j}E_{i,j}$, where the sum is over all pairs of boxes $B_i$, $B_j$ in the orthosymplectic Dynkin pyramid satisfying one of the following:
\begin{itemize}
  \item $row(B_i) = row(B_j)$ and $col(B_i) = col(B_j) + 2$;
  \item $row(B_i) = -row(B_j)$ is an even skew-row in the upper half plane, $col(B_i) =
2, col(B_j) = 0$;
  \item $row(B_i) = -row(B_j)$ is an even skew-row in the upper half plane, $col(B_i) =
0, col(B_j) = -2$;
  \item $row(B_i) = -row(B_j)$ is an odd skew-row in the upper half plane, $col(B_i) =
1, col(B_j) = -1$,
\end{itemize}
where $row(B_i)$ (resp. $col(B_i)$) denotes the first (resp. second) coordinate of the box $B_i$.
Then $e_{p,q}$ is a nilpotent element of $\mathfrak{osp}(m|2n)$ and corresponds to the partition $(p|q)$.

Define $h_{p,q}$ to be the $(m+2n)$-diagonal matrix whose eigenvalue on the vector $v_i$ is equal to the first coordinate of the box labeled with this vector. Then $h_{p,q}$ defines a $\mathbb{Z}$-grading of $\mathfrak{g}$ for which $e_{p,q}\in\mathfrak{g}(2)$.  This is the Dynkin grading for $e_{p,q}$.

\begin{example}
$\mathfrak{osp}(6,4)$. The pyramids for the partitions $(3,3 | 4)$, $(5,1 | 2,2)$ are:\\
\setlength{\unitlength}{.75cm}
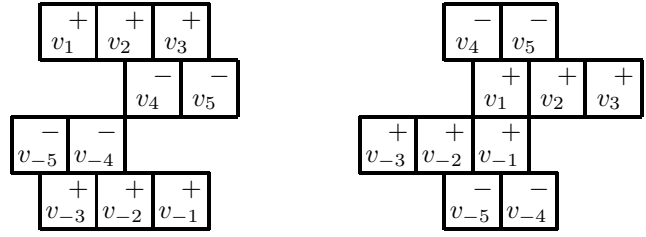
\begin{center}
\begin{picture}(6,4)(0,0)
\linethickness{1pt}
\put(.5,0){\line(1,0){3}}
\put(0,1){\line(1,0){3.5}}
\put(0,2){\line(1,0){4}}
\put(0.5,3){\line(1,0){3.5}}
\put(0.5,4){\line(1,0){3}}
\put(0,1){\line(0,1){1}}
\put(0.5,0){\line(0,1){1}}
\put(0.5,3){\line(0,1){1}}
\put(1,1){\line(0,1){1}}
\put(1.5,0){\line(0,1){1}}
\put(1.5,3){\line(0,1){1}}
\put(2,1){\line(0,1){2}}
\put(2.5,0){\line(0,1){1}}
\put(2.5,3){\line(0,1){1}}
\put(3,2){\line(0,1){1}}
\put(3.5,0){\line(0,1){1}}
\put(3.5,3){\line(0,1){1}}
\put(4,2){\line(0,1){1}}
\put(0.6,0.2){$v_{-3}$}
\put(1.6,0.2){$v_{-2}$}
\put(2.6,0.2){$v_{-1}$}
\put(0.1,1.2){$v_{-5}$}
\put(1.1,1.2){$v_{-4}$}
\put(0.7,3.2){$v_{1}$}
\put(1.7,3.2){$v_{2}$}
\put(2.7,3.2){$v_{3}$}
\put(3.2,2.2){$v_{5}$}
\put(2.2,2.2){$v_{4}$}
\put(1,0.6){$+$}
\put(2,0.6){$+$}
\put(3,0.6){$+$}
\put(0.5,1.6){$-$}
\put(1.5,1.6){$-$}
\put(1,3.6){$+$}
\put(2,3.6){$+$}
\put(3,3.6){$+$}
\put(3.5,2.6){$-$}
\put(2.5,2.6){$-$}
\end{picture}
\begin{picture}(6,4)(0,0)
\linethickness{1pt}
\put(1.5,0){\line(1,0){2}}
\put(0,1){\line(1,0){3.5}}
\put(0,2){\line(1,0){5}}
\put(1.5,3){\line(1,0){3.5}}
\put(1.5,4){\line(1,0){2}}
\put(0,1){\line(0,1){1}}
\put(1,1){\line(0,1){1}}
\put(1.5,0){\line(0,1){1}}
\put(1.5,3){\line(0,1){1}}
\put(2,1){\line(0,1){2}}
\put(2.5,0){\line(0,1){1}}
\put(2.5,3){\line(0,1){1}}
\put(3,1){\line(0,1){2}}
\put(3.5,0){\line(0,1){1}}
\put(3.5,3){\line(0,1){1}}
\put(4,2){\line(0,1){1}}
\put(5,2){\line(0,1){1}}
\put(1.6,0.2){$v_{-5}$}
\put(2.6,0.2){$v_{-4}$}
\put(0.1,1.2){$v_{-3}$}
\put(1.1,1.2){$v_{-2}$}
\put(2.1,1.2){$v_{-1}$}
\put(2.2,2.2){$v_1$}
\put(3.2,2.2){$v_2$}
\put(4.2,2.2){$v_3$}
\put(1.7,3.2){$v_4$}
\put(2.7,3.2){$v_5$}
\put(2,0.6){$-$}
\put(3,0.6){$-$}
\put(0.5,1.6){$+$}
\put(1.5,1.6){$+$}
\put(2.5,1.6){$+$}
\put(2.5,2.6){$+$}
\put(3.5,2.6){$+$}
\put(4.5,2.6){$+$}
\put(2,3.6){$-$}
\put(3,3.6){$-$}
\end{picture}
\end{center}
\end{example}

Let $$C(p):=\{p_i\in J_p \mid p_i \text{ is odd, } m_{p_i}=2 \text{ and } p_i\not\in J_q\}=\{p_1>\cdots >p_{c(p)}\}$$ and $$D(q):=\{q_j\in J_q \mid q_j \text{ is even, } m_{q_j}=2 \text{ and }q_j\not\in J_p\}=\{q_1>\cdots >q_{d(q)}\}.$$

Define the diagonal matrices $z(s_1,\ldots,s_{c(p)})\in\mathfrak{so}(m)$, with $s_i\in\mathbb{F}$, whose $i^{th}$ diagonal entry is $s_i$ if the basis vector lies in a box of $SP(p)$ in the (strictly) upper half-plane in a row corresponding to the part $p_i\in C(p)$, and is $-s_i$ if the basis vector lies in the centrally symmetric box, and all other entries are zero.
Define the diagonal matrices $z(t_1,\ldots,t_{d(q)})\in\mathfrak{sp}(2n)$, with $t_j\in\mathbb{F}$, whose $j^{th}$ diagonal entry is $t_j$ if the basis vector lies in a box of $SP(p)$ in the (strictly) upper half-plane in a row corresponding to the part $q_j\in D(q)$, and is $-t_j$ if the basis vector lies in the centrally symmetric box, and all other entries are zero.

\begin{theorem}\label{thmosp}
Let $\mathfrak{g}=\mathfrak{osp}(m|2n)$ and let $(p,q)$ be an orthosymplectic partition of $(m|n)$.

If m=2k+1, the element $h_{p,q} + (z(s_1,\ldots,s_{c(p)}),z(t_1,\ldots,t_{d(q)}))$ defines a good $\mathbb{Z}$-grading of $\mathfrak{osp}(2k+1|2n)$ for $e_{p,q}$ if and only if one of the following cases holds:
\begin{enumerate}[(i)]
  \item if $1\not\in C(p)$, then $s_i,t_j\in\{-1,0,1\}$ for $1\leq i\leq c(p)$, $1\leq j\leq d(q)$, and for each pair $p_k\in C(p)$, $q_l\in D(q)$ satisfying $p_k=q_l \pm 1$ we must have $|s_k -t_l|\leq 1$;

  \item if $1\in C(p)$, then $s_i,t_j\in\{-1,0,1\}$ for $1\leq i\leq c(p)-1$, $1\leq j\leq d(q)$, $s_{c(p)}\in\mathbb{Z}$, and for each pair $p_k\in C(p)$, $q_l\in D(q)$ satisfying $p_k=q_l \pm 1$ we must have $|s_k -t_l|\leq 1$, and
      $$|s_{c(p)}|\ \leq\  \mbox{min}\{p_{\alpha-1}-1,\ q_{\beta}-1,\ p_{c(p)-1}-|s_{c(p)-1}|-1,\ q_{d(q)}-|t_{d(q)}|-1\}.$$
\end{enumerate}

If m=2k, the element $h_{p,q} + (z(s_1,\ldots,s_{c(p)}),z(t_1,\ldots,t_{d(q)}))$ defines a good $\mathbb{Z}$-grading of $\mathfrak{osp}(2k|2n)$ for $e_{p,q}$ if and only if one of the following cases holds:
\begin{enumerate}[(i)]
  \item if $1\not\in C(p)$, and $C(p)\neq J_p$ or $D(q)\neq J_q$, then $s_i,t_j\in\{-1,0,1\}$ for $1\leq i\leq c(p)$, $1\leq j\leq d(q)$, and for each pair $p_k\in C(p)$, $q_l\in D(q)$ satisfying $p_k=q_l \pm 1$ we must have $|s_k -t_l|\leq 1$;

   \item if $1\not\in C(p)$ and $C(p)=J_p$, $D(q)=J_q$, then either all $s_i,t_j\in\{-1,0,1\}$ for $1\leq i\leq c(p)$, $1\leq j\leq d(q)$ or all $s_i,t_j\in\{-1/2,1/2\}$, and for each pair $p_k\in C(p)$, $q_l\in D(q)$ satisfying $p_k=q_l \pm 1$ we must have $|s_k -t_l|\leq 1$;

  \item if $1\in C(p)$, and $C(p)\neq J_p$ or $D(q)\neq J_q$, then $s_i,t_j\in\{-1,0,1\}$ for $1\leq i\leq c(p)-1$, $1\leq j\leq d(q)$, $s_{c(p)}\in\mathbb{Z}$, and for each pair $p_k\in C(p)$, $q_l\in D(q)$ satisfying $p_k=q_l \pm 1$ we must have $|s_k -t_l|\leq 1$, and
      $$|s_{c(p)}|\ \leq\  \mbox{min}\{p_{\alpha-1}-1,\ q_{\beta}-1,\ p_{c(p)-1}-|s_{c(p)-1}|-1,\ q_{d(q)}-|t_{d(q)}|-1\}.$$

   \item if $1\in C(p)$ and $C(p)=J_p$, $D(q)=J_q$, then either all $s_i,t_j\in\{-1,0,1\}$ for $1\leq i\leq c(p)-1$, $1\leq j\leq d(q)$, $s_{c(p)}\in\mathbb{Z}$, or all $s_i,t_j\in\{-1/2,1/2\}$ and $s_{c(p)}\in (1/2+\mathbb{Z})$, and for each pair $p_k\in C(p)$, $q_l\in D(q)$ satisfying $p_k=q_l \pm 1$ we must have $|s_k -t_l|\leq 1$, and
       $$|s_{c(p)}|\ \leq\  \mbox{min}\{p_{\alpha-1}-1,\ q_{\beta}-1,\ p_{c(p)-1}-|s_{c(p)-1}|-1,\ q_{d(q)}-|t_{d(q)}|-1\}.$$
      \end{enumerate}
\end{theorem}

\begin{proof}
Let $e=e_{p,q}$, $h=h_{p,q}$ and let $\mathfrak{s}=\{e,h,f\}$ be the corresponding $\mathfrak{sl}_{2}$-triple.
As in the $\mathfrak{gl}(m|n)$ case, all good $\mathbb{Z}$-gradings for $e_{p,q}$ can be describe by the eigenvalues and eigenspaces of $\mbox{ad}(h_{p,q}+z)$ for some semisimple element $z\in\mathfrak{Z}(\mathfrak{g}^{s})$.  Recall that $\mathfrak{g}_{\bar{0}}=\mathfrak{so}(m)\times\mathfrak{sp}(2n)$.  The conditions on the diagonal matrix $z$ which imply $\mathfrak{g}_{\bar{0}}^{e}\subset(\mathfrak{g}_{\bar{0}})_{\geq}$ where determined in \cite{EK}.  So we only need to determine the additional conditions which imply  $\mathfrak{g}_{\bar{1}}^{e}\subset(\mathfrak{g}_{\bar{1}})_{\geq}$. Now $\mathfrak{g}_{\bar{1}}\cong \mbox{Hom}(V_{0},V_{1})$, so these conditions are the same as for the odd part of the $\mathfrak{gl}(m|n)$ case.  In particular, for all $i,j$ we must have $s_i-t_j \in\mathbb{Z}$ and $|s_i-t_j| \leq |p_i-q_j|$ where $s_i$ (resp. $t_j$) corresponds to the partition part $p_i$ (resp. $q_j$).
\end{proof}

\end{document}